\documentclass{amsart}
\usepackage{verbatim}
\usepackage{amssymb}
\usepackage{amsmath}
\usepackage[usenames]{color}
\usepackage{graphicx}
\usepackage{amscd}
\usepackage[numbers,sort,square]{natbib}

\usepackage{enumerate}

\allowdisplaybreaks
\theoremstyle{plain}
\newtheorem{theorem}{Theorem}

\newtheorem{proposition}[theorem]{Proposition}
\newtheorem{lemma}[theorem]{Lemma}
\newtheorem{corollary}[theorem]{Corollary}

\newtheorem{definition}[theorem]{Definition}

\theoremstyle{definition}

\newtheorem{remark}[theorem]{Remark}

\numberwithin{equation}{section}
\numberwithin{theorem}{section}

\newcommand{\eps}{\varepsilon}

\newcommand{\sign}{\mathrm{sign}\text{ }}
\newcommand{\Var}{{\rm Var}\,}

\newcommand{\ud}[0]{\,\mathrm{d}}

\newcommand{\vertiii}[1]{{\left\vert\kern-0.25ex\left\vert\kern-0.25ex\left\vert #1 
    \right\vert\kern-0.25ex\right\vert\kern-0.25ex\right\vert}}

\begin{document}

\title[On the martingale decompositions in infinite dimensions]
{On the martingale decompositions of Gundy, Meyer, and Yoeurp in infinite dimensions}

\author{Ivan Yaroslavtsev}
\address{Delft Institute of Applied Mathematics\\
Delft University of Technology \\ P.O. Box 5031\\ 2600 GA Delft\\The
Netherlands}
\email{I.S.Yaroslavtsev@tudelft.nl}

\begin{abstract}
We show that the canonical decomposition (comprising both the Meyer-Yoeurp and the Yoeurp decompositions) of a general $X$-valued local martingale is possible if and only if $X$ has the UMD property. More precisely, $X$ is a UMD Banach space if and only if for any $X$-valued local martingale $M$ there exist a continuous local martingale $M^c$, a purely discontinuous quasi-left continuous local martingale $M^q$, and a purely discontinuous local martingale $M^a$ with accessible jumps such that $M = M^c + M^q + M^a$. The corresponding weak $L^1$-estimates are provided. Important tools used in the proof are a~new version of Gundy's decomposition of continuous-time martingales and weak $L^1$-bounds for a certain class of vector-valued continuous-time martingale transforms.
\end{abstract}



\maketitle

\section{Introduction}

It is well-known thanks to the scalar-valued stochastic integration theory that a~stochastic integral $\int \Phi\ud N$ of a general bounded predictable real-valued process $\Phi$ with respect to a general real-valued local martingale $N$ exists and is well defined (see e.g.\ Chapter 26 in \cite{Kal}). Moreover, $\int \Phi\ud N$ is a local martingale, so by the Burkholder-Davis-Gundy inequalities one can show the corresponding $L^p$-estimates for $p\in (1,\infty)$:
\begin{equation}\label{eq:introBDGscalar}
  \mathbb E \sup_{0\leq s\leq t}\Bigl|\int_0^s \Phi\ud N\Bigr|^p \eqsim_p \mathbb E \Bigl(\int_0^t \Phi^2 \ud [N]\Bigr)^{\frac p2},\;\;\; t\geq 0
\end{equation}
(here $[N]:\mathbb R_+\times \Omega \to \mathbb R_+$ is a quadratic variation of $N$, see \eqref{eq:defquadvar} for the definition).
The inequality \eqref{eq:introBDGscalar} together with a Banach fixed point argument play an important r\^ole in providing solutions to SPDE's with a general martingale noise (see e.g.\ \cite{VY2016,Kal,DPZ,NVW,GK1,G3} and references therein). For this reason \eqref{eq:introBDGscalar}-type inequalities for a broader class of $N$ and $\Phi$ are of interest. In particular, one can consider $H$-valued $N$ and $\mathcal L(H,X)$-valued $\Phi$ for some Hilbert space $H$ and Banach space $X$. Building on ideas of Garling \cite{Gar85} and McConnell \cite{MC}, van Neerven, Veraar, and Weis have shown in \cite{NVW} that for a special choice of $N$ (namely, $N$ being a Brownian motion) and a general process $\Phi$ it is necessary and sufficient that $X$ is in the class of so-called {\em UMD Banach spaces} (see Subsection~\ref{subsec:UMD} for the definition) in order to obtain estimates of the form~\eqref{eq:introBDGscalar} with the right-hand side replaces by a generalized square function. Later in the paper \cite{Ver} by Veraar and in the paper \cite{VY2016} by Veraar and the author, inequalities of the form \eqref{eq:introBDGscalar} have been extended to a general continuous martingale $N$, again given that $X$ has the UMD property.

Extending \eqref{eq:introBDGscalar} to a general martingale $N$ is an open problem, which was solved only for $X = L^q(S)$ with $q\in (1,\infty)$ in the recent work \cite{DY17} by Dirksen and the author. One of the key tools applied therein was the so-called {\em canonical decomposition} of martingales. The canonical decomposition first appeared in the work \cite{Yoe76} by Yoeurp, and partly in the paper \cite{Mey76} by Meyer, and has the following form: an $X$-valued local martingale $ M$ is said to admit the canonical decomposition if there exists a continuous local martingale $ M^c$, a purely discontinuous quasi-left continuous local martingale $ M^q$ (a ``Poisson-like'' martingale which does not jump at predictable stopping times), and a purely discontinuous local martingale $ M^a$ with accessible jumps (a ``discrete-like'' martingale which jumps only at a certain countable set of predictable stopping times) such that $ M^c_0= M^q_0=0$ a.s.\ and $ M=  M^c +  M^q +  M^a$. The canonical decomposition (if it exists) is unique due to the uniqueness in the case $X = \mathbb R$ (see Remark \ref{rem:MY==weakMY} and \ref{rem:Y==weakY}). Moreover, when $X$ is UMD one has by \cite{Y17MartDec} that for all $p\in (1,\infty)$,
\begin{equation}\label{eq:introcandecL^pest}
  \mathbb E \| M_t\|^p \eqsim_{p,X}\mathbb E \| M^c_t\|^p +\mathbb E \| M^q_t\|^p+\mathbb E \| M^a_t\|^p,\; t\geq 0.
\end{equation}
In particular, if $N$ is $H$-valued and $\Phi$ is $\mathcal L(H, X)$-valued, then 
$$
\int \Phi \ud N = \int \Phi \ud N^c+ \int \Phi \ud N^q + \int \Phi \ud N^a
$$ 
is the canonical decomposition given that $N = N^c + N^q + N^a$ is the canonical decomposition, so
\[
 \mathbb E \Bigl\|\int_0^t \Phi \ud N\Bigr\|^p \eqsim_{p,X}\mathbb E \Bigl\|\int_0^t \Phi \ud N^c\Bigr\|^p +\mathbb E \Bigl\|\int_0^t \Phi \ud N^q\Bigr\|^p+\mathbb E \Bigl\|\int_0^t \Phi \ud N^a\Bigr\|^p,\;\;\; t\geq 0,
\]
which together with Doob's maximal inequality reduces the problem of extending \eqref{eq:introBDGscalar} to the separate cases of $N^c$, $N^q$ and $N^a$. Possible approaches of how to work with $\int \Phi \ud N^c$, $\int \Phi \ud N^q$, and $\int \Phi \ud N^a$ have been provided by \cite{DY17}: sharp estimates for the first were already obtained in \cite{Ver,VY2016} and follow from the similar estimates for a Brownian motion from \cite{NVW}; the second can be treated by using random measure theory (see Subsection \ref{subsec:randommeasures}), which is an extension of Poisson random measure integration theory (see \cite{Dirk14} and \cite{DMN}); finally, the latter one can be transformed to a discrete martingale by an approximation argument, so the desired $L^p$-estimates are nothing more but the {\em Burkholder-Rosenthal inequalities} (see \cite{DY17,Bur73,Ros70} for details). 

\smallskip

The canonical decomposition also plays a significant r\^ole in obtaining $L^p$-estimates for {\em weakly differentially subordinated} martingales. The weak differential subordination property as a vector-valued generalization of Burkholder's differential subordination property (see \cite{Burk84,Os12,HNVW1,KW92}) was introduced by the author in \cite{Y17FourUMD}, and can be described in the following way: an $X$-valued local martingale $\widetilde M$ is weakly differentially subordinated to an \mbox{$X$-va}\-lued local martingale $M$ if for each $x^* \in X^*$ and for each $t\geq s\geq 0$ a.s.
\begin{align*}
 |\langle\widetilde  M_0, x^*\rangle| &\leq |\langle M_0, x^*\rangle|,\\
 [\langle\widetilde  M, x^*\rangle]_t - [\langle\widetilde  M, x^*\rangle]_s &\leq [\langle M, x^*\rangle]_t - [\langle M, x^*\rangle]_s.
\end{align*}
If $X$ is a UMD Banach space and $p\in (1,\infty)$, then applying $L^p$-bounds \eqref{eq:introcandecL^pest} for the terms of the canonical decomposition together with $L^p$-bounds for purely discontinuous (see \cite{Y17FourUMD}) and continuous (see \cite{Y17MartDec}) weakly differentially subordinated martingales yields
\begin{equation}\label{eq:introL^pestWDS}
  (\mathbb E \|\widetilde M_{\infty}\|^p)^{\frac 1p} \leq c_{p,X} (\mathbb E \|M_{\infty}\|^p)^{\frac 1p},
\end{equation}
where the best known constant $c_{p,X}$ equals $\beta_{p, X}^2(\beta_{p, X}+1)$ (here $\beta_{p,X}$ is the {\em UMD$_p$ constant of $X$}, see Subsection \ref{subsec:UMD} for the definition). Sharp estimates for $c_{p,X}$ in \eqref{eq:introL^pestWDS} remain unknown. Moreover, it is an open problem whether one can prove weak $L^1$-estimates of the form
\begin{equation}\label{eq:introweakL^1estWDS}
  \lambda \mathbb P\bigl(\widetilde M^*_{\infty}>\lambda\bigr)\lesssim_{p,X} \mathbb E \|M_{\infty}\|,\;\;\; \lambda>0.
\end{equation}
Here this question is partly solved: we show that \eqref{eq:introweakL^1estWDS} holds for $\widetilde M$ being one of the terms of the canonical decomposition of $M$ (see \eqref{eq:introweakL^1forcanonical} and \eqref{eq:L^1inftyforcor:candecoflocmartsufofUMD}).

\smallskip

The discussion above demonstrates that the canonical decomposition is useful for vector-valued stochastic integration and weak differential subordination, so the following natural question arises: {\em for which Banach spaces $X$ does every $X$-valued local martingale have the canonical decomposition?} The paper \cite{Y17MartDec} together with the estimates \eqref{eq:introcandecL^pest} provides the answer for $L^p$-martingales given $p\in(1,\infty)$. Then $X$ being a UMD Banach space guarantees such a decomposition.

The present paper is devoted to providing the definitive answer to this question (see Section \ref{sec:candecoflocalmartUMD}):

\begin{theorem}\label{thm:intromainwithweakL^1est}
 Let $X$ be a Banach space. Then the following are equivalent:
 \begin{enumerate}[(i)]
  \item $X$ is a UMD Banach space;
  \item every local martingale $M:\mathbb R_+ \times \Omega \to X$ admits the canonical decomposition $M = M^c + M^q + M^a$. 
 \end{enumerate}
Moreover, if this is the case, then for all $t\geq 0$ and $\lambda>0$
\begin{equation}\label{eq:introweakL^1forcanonical}
 \begin{split}
  \lambda \mathbb P (( M^c)^*_t >\lambda) &\lesssim_X \mathbb E \| M_t\|,\\
   \lambda \mathbb P (( M^q)^*_t >\lambda) &\lesssim_X \mathbb E \| M_t\|,\\
    \lambda \mathbb P (( M^a)^*_t >\lambda) &\lesssim_X \mathbb E \| M_t\|.
 \end{split}
\end{equation}
\end{theorem}
Notice that the inequalities \eqref{eq:introweakL^1forcanonical} are new even in the real-valued case, even though in that case they are direct consequences of the sharp weak $(1,1)$-estimates for differentially subordinated martingales proven by Burkholder in \cite{Burk87,Burk89} (see also \cite{Os12,Os07a} for details), from which one can show the following estimates
\begin{equation*}
 \begin{split}
  \lambda \mathbb P (( M^c)^*_t >\lambda) &\leq 2 \mathbb E | M_t|,\\
   \lambda \mathbb P (( M^q)^*_t >\lambda) &\leq 2 \mathbb E | M_t|,\\
    \lambda \mathbb P (( M^a)^*_t >\lambda) &\leq 2 \mathbb E | M_t|.
 \end{split}
\end{equation*}

The main instrument for proving ${(ii)}\Rightarrow {(i)}$ in Theorem \ref{thm:intromainwithweakL^1est} is Burkholder's characterization of UMD Banach spaces from \cite{Burk81}: $X$ is a UMD Banach space if and only if there exists a constant $C>0$ such that for any $X$-valued discrete martingale $(f_n)_{n\geq 0}$, for any sequence $(a_n)_{n\geq 0}$ with values in $\{-1,1\}$ one has that
 \[
  g^*_{\infty} >1\;\; \text{a.s.}\;\;\Longrightarrow \;\;\mathbb E\|f_{\infty}\|>C,
 \]
where $(g_n)_{n\geq 0}$ is an $X$-valued discrete martingale such that 
\begin{equation}\label{eq:introBurkmarttrans}
 \begin{split}
   g_n - g_{n-1} &= a_n(f_n-f_{n-1}),\;\;\;n\geq 1,\\
g_0&=a_0f_0,
 \end{split}
\end{equation}
and where $g^*_{\infty} := \sup_{n\geq 0} \|g_n\|$. Using this characterization for a given non-UMD Banach space $X$ we construct a martingale $M:\mathbb R_+ \times \Omega \to X$ which does not have the canonical decomposition (see Subsection \ref{subsec:NecofUMD}).

In order to obtain weak $L^1$-estimates of the form \eqref{eq:introweakL^1forcanonical} together with $(i)\Rightarrow (ii)$ in Theorem \ref{thm:intromainwithweakL^1est} one needs to use two techniques.
The first is the so-called {\em Gundy decomposition} of martingales.
This decomposition was first obtained by Gundy in \cite{Gundy68} for discrete real-valued martingales.
Later in \cite{HNVW1,MT00,BG70,PR06} a more general version of this decomposition for vector-valued discrete martingales was obtained. 
In Section~\ref{sec:Gundy'sdec} we will present a continuous-time analogue of Gundy's decomposition, which has the following form:
an $X$-valued martingale $M$ can be decomposed into a sum of three martingales $M^1$, $M^2$, and $M^3$, depending on $\lambda>0$, such that for each $t\geq 0$
\begin{itemize}
  \item [(i)] $\|M^1_{t}\|_{L^{\infty}(\Omega; X)}\leq 2\lambda$, $\mathbb E\|M_t^1\|\leq 5 \mathbb E\|M_t\|$,
  \item [(ii)] $\lambda \mathbb P ((M^2)^*_t>0)\leq 4\mathbb E\|M_t\|$,
  \item [(iii)] $\mathbb E(\Var M^3)_t\leq 7 \mathbb E\|M_t\|$,
 \end{itemize}
where $\Var M$ is a variation of the path of $M$.

\smallskip

The second important tool is {\em weak differential subordination martingale transforms}. Discrete martingale transforms were pioneered by Burkholder in~\cite{Burk66}, where he considered a transform $(f_n)_{n\geq 0}\mapsto (g_n)_{n\geq 0}$ of a real-valued martingale $(f_n)_{n\geq 0}$ such that
\begin{align*}
 g_n-g_{n-1}&=a_n (f_n-f_{n-1}),\;\;\; n\geq 1,\\
 g_0&=a_0f_0
\end{align*}
for some $\{0,1\}$-valued deterministic sequence $(a_n)_{n\geq 0}$. Later in \cite{Burk81,Hit90,HNVW1,MT00,BG70,GW05} several approaches and generalizations to the vector-valued setting and operator-valued predictable sequence $(a_n)_{n\geq 0}$ have been discovered, while the martingale $(f_n)_{n\geq 0}$ remained discrete. In particular for a very broad class of discrete martingale transforms it was shown that $L^p$-boundedness of the transform implies weak $L^1$-bounds. In Subsection \ref{subsec:WDSoperator} (see Theorem \ref{thm:analougeofProp3.5.4}) we prove the same assertion for a~weak differential subordination martingale transform, i.e.\ for an operator $T$ acting on continuous-time $X$-valued local martingales such that $TM$ is weakly differentially subordinated to $M$ and $\{M^*_{\infty}=0\}\subset \{(TM)^*_{\infty}=0\}$ for any $X$-valued local martingale $M$. A particular example of such a martingale transform $T$ is $M\mapsto TM = M^c$, where $M^c$ is the continuous part of $M$ in the canonical decomposition. Due to \eqref{eq:introcandecL^pest} this operator is bounded as an operator acting on $L^p$-martingales if $X$ is UMD, so by Theorem~\ref{thm:analougeofProp3.5.4} the first inequality of \eqref{eq:introweakL^1forcanonical} follows.
Even though in the case of a discrete filtration such an operator has a classical Burkholder's form \eqref{eq:introBurkmarttrans} from \cite{Burk81} (with $(a_n)_{n\geq 0}$ being predictable instead of deterministic, see Proposition \ref{prop:discretefiltrationopProp3.5.4} and the remark thereafter), such transforms are of interest since they act on continuous-time martingales, which was not considered before.

\smallskip

\emph{Acknowledgment} -- The author would like to thank Mark Veraar for useful suggestions and fruitful discussions. The author thanks Alex Amenta and Jan van Neerven for careful reading of parts of this article and helpful comments.

\section{Preliminaries}

In the sequel the scalar field is assumed to be $\mathbb R$, unless stated otherwise.

\subsection{UMD Banach spaces}\label{subsec:UMD}\nopagebreak
A Banach space $X$ is called a {\em UMD space} if for some (or equivalently, for 
all)
$p \in (1,\infty)$ there exists a constant $\beta>0$ such that
for every $N \geq 1$, every martingale
difference sequence $(d_n)^N_{n=1}$ in $L^p(\Omega; X)$, and every scalar-valued 
sequence
$(\varepsilon_n)^N_{n=1}$ such that $|\varepsilon_n|=1$ for each $n=1,\ldots,N$
we have
\[
\Bigl(\mathbb E \Bigl\| \sum^N_{n=1} \varepsilon_n d_n\Bigr\|^p\Bigr )^{\frac 
1p}
\leq \beta \Bigl(\mathbb E \Bigl \| \sum^N_{n=1}d_n\Bigr\|^p\Bigr )^{\frac 1p}.
\]
The least admissible constant $\beta$ is denoted by $\beta_{p,X}$ and is called 
the {\em UMD$_p$~constant} or, if the value of $p$ is understood, the {\em UMD constant}, of $X$.
It is well-known that UMD spaces obtain a large number 
of useful properties, such as being reflexive. Examples of UMD 
spaces include all finite dimensional spaces and the reflexive range of 
$L^q$-spaces, Besov spaces, Sobolev spaces, Schatten class spaces, and Orlicz spaces. Example of 
spaces without the UMD property include all nonreflexive Banach spaces, e.g.\ 
$L^1(0,1)$ and $C([0,1])$. We refer the reader to 
\cite{Burk01,HNVW1,Rubio86,Pis16} for details.

\subsection{Martingales and c\`adl\`ag processes}\label{subsec:martandcadlagprocesses}
Let $X$ be a Banach space, $\mathbb F = (\mathcal F_t)_{t\geq 0}$ be a filtration that satisfies the usual conditions (e.g.\ right-continuity). For each $1\leq p< \infty$ a martingale $M:\mathbb R_+ \times \Omega \to X$ is called an {\em $L^{p}$-martingale} (or, an {\em $L^p$-integrable} martingale) if $M_{\infty} := \lim_{t\to \infty} M_{t}$ exists in $L^p(\Omega; X)$; we call $M$ an {\em $L^{\infty}$-martingale} if $\|M_t\|_{L^{\infty}(\Omega; X)}$ is uniformly bounded in $t\in \mathbb R_+$. For a given $p\in [1,\infty)$ we will denote the set of all $X$-valued $L^p$-integrable $\mathbb F$-martingales by $\mathcal M_X^{p}(\mathbb F)$; further, we will denote the set of all $X$-valued local $L^p$-integrable \mbox{$\mathbb F$-mar}\-tin\-gales by $\mathcal M_X^{p,\rm loc}(\mathbb F)$. Note that $\mathcal M_X^{p}(\mathbb F)$ is a Banach space endowed with the norm $\|M\|_{\mathcal M_X^{p}(\mathbb F)} := \|M_{\infty}\|_{L^p(\Omega; X)}$.

We will denote by $\mathcal M_X^{1,\infty}(\mathbb F)$ the set of all $X$-valued local $\mathbb F$-martingales $M$ such that
\[
 \sup_{\lambda>0}\lambda \mathbb P(M^*_{\infty}>\lambda)<\infty.
\]
In the sequel we will omit $\mathbb F$ from the notations $\mathcal M_X^{p}(\mathbb F)$, $\mathcal M_X^{p,\rm loc}(\mathbb F)$, and $\mathcal M_X^{1,\infty}(\mathbb F)$.

\begin{remark}
 Let $X$ be a Banach space, $M:\mathbb R_+ \times \Omega \to X$ be a martingale. Then $(N_t)_{t\geq 0} := (\|M_t\|)_{t\geq 0}$ is a submartingale by \cite[Lemma 7.11]{Kal} and the fact that $x\mapsto \|x\|$ is a convex function on $X$. Moreover, by \cite[Theorem 1.3.8(i)]{KS} we have that for each $t\geq 0$, $p\geq 1$ and $\lambda >0$
 \begin{equation}\label{eq:1,inftyineqforsubmart}
  \mathbb P(M^*_t>\lambda)\leq \frac{\mathbb E \|M_t\|^p}{\lambda^p}.
 \end{equation}
\end{remark}

\smallskip

A function $f:\mathbb R_+ \to X$ is called c\`adl\`ag (a French abbreviation of the phrase ``continuous from right, limits from left'') if it is right-continuous and if it has left-hand limits.
A process $V:\mathbb R_+ \times \Omega \to X$ is called c\`adl\`ag if it has c\`adl\`ag paths. For instance, any martingale $M:\mathbb R_+ \times \Omega \to X$ has a c\`adl\`ag version given $\mathbb F$ satisfies the usual assumptions (see \cite{Y17FourUMD} for details in the vector-valued setting).

Let $\tau$ be a stopping time. If $V:\mathbb R_+ \times \Omega \to X$ is c\`adl\`ag, then we can define $\Delta V_{\tau}:\Omega \to X$ in the following way:
\[
 \Delta V_{\tau} = 
 \begin{cases}
  V_0,\;\;\; &\tau =0,\\
 V_{\tau} - \lim_{\eps\to 0}V_{0\vee(\tau-\eps)},\;\;\; &0<\tau < \infty,\\
 0,\;\;\;&\tau = \infty,
\end{cases}
\]
where $\lim_{\eps\to 0}V_{0\vee(\tau-\eps)}$ exists since $V$ has paths with left-hand limits.

\smallskip

One can define the so-called {\em ucp topology} (uniform convergence on compact sets in probability) on the linear space of all c\`adl\`ag adapted $X$-valued processes; convergence in this topology can be characterized in the following way: a sequence $(V^n)_{n\geq 1}$ of c\`adl\`ag adapted $X$-valued processes converges to $V:\mathbb R_+ \times \Omega \to X$ in the ucp topology if for any $t\geq 0$ and $K>0$ we have that
\begin{equation}\label{eq:ucpconvergencehow}
  \mathbb P \bigl(\sup_{0\leq s\leq t} \|V_s-V^n_s\|>K\bigr) \to 0\;\;\;\; n\to \infty.
\end{equation}
%, which is generated by all the seminorms, defined :
%\[
% \rho_{\rm ucp} (U, V) := \sum_{n= 1}^{\infty} \frac{1}{2^n}\mathbb E (\sup_{0\leq s\leq n} \|U_s-V_s\|\wedge 1),
%\]
%where $U,V:\mathbb R_+ \times \Omega \to X$ are c\`adl\`ag adapted processes.
Then the following proposition holds.

\begin{proposition}\label{prop:Xvaluedcadlagareclosedunderucp}
 The linear space of all c\`adl\`ag adapted $X$-valued processes endowed with the ucp topology is complete.
\end{proposition}

\begin{proof}
 This is just the vector-valued analogue of \cite[Theorem 62]{Prot}, for which one needs to apply the vector-valued variation of \cite[Problem V.1]{Pollard}.
\end{proof}

\smallskip

We state without proof the following elementary but useful statement.

\begin{lemma}\label{lem:supofcontfuncwithlimitiscont}
 Let $X$ be a Banach space, $(f_n)_{n\geq 1}$, $f$ be continuous \mbox{$X$-va}\-lu\-ed functions on $[0,1]$ such that $f_n \to f$ in $C([0,1]; X)$ as $n\to \infty$. Then the function $F:[0,1]\to\mathbb R_+$ defined as follows
 \[
  F(t) = \sup_n\|f_n(t)\|,\;\;\; t\in [0,1],
 \]
is continuous.
\end{lemma}

\subsection{Purely discontinuous martingales}\label{subsec:pdmart}
Let $M:\mathbb R_+\times \Omega \to \mathbb R$ be a local martingale. Then $M$ 
is called {\em purely discontinuous} if $[M]$ is a pure jump process (i.e.\ 
$[M]$ has a version that is a constant a.s.\ in time). Let $X$ be a Banach 
space, $M:\mathbb R_+\times \Omega \to X$ be a local martingale. Then $M$ is 
called {\em purely discontinuous} if for each $x^* \in X^*$ a local martingale 
$\langle M, x^*\rangle$ is purely discontinuous. The following proposition can be found in \cite{Y17MartDec,DY17}.

\begin{proposition}\label{thm:purdiscorthtoanycont1}
  A martingale $M:\mathbb R_+\times \Omega \to X$ is purely discontinuous if and only if $M N$ is a martingale for any continuous bounded martingale $N:\mathbb R_+ \times \Omega \to \mathbb R$ such that $N_0 = 0$. 
 \end{proposition}
 
In the sequel we will use the following lemma, which proof can be found in \cite{Y17MartDec,DY17}.
\begin{lemma}\label{lemma:contpuredisczero}
 Let $X$ be a Banach space, $M:\mathbb R_+ \times \Omega \to X$ be a martingale which is both continuous and purely discontinuous. Then $M = M_0$ a.s.
\end{lemma}

 The reader can find more on purely discontinuous martingales in \cite{Y17MartDec,DY17,JS,Jac79,Kal,Y17FourUMD}.

\subsection{Random measures}\label{subsec:randommeasures}

Let $(J, \mathcal J)$ be a measurable space. Then a family $\mu = \{\mu(\omega; \ud t, \ud x), \omega \in \Omega\}$ of nonnegative measures on $(\mathbb R_+ \times J; \mathcal B(\mathbb R_+)\otimes \mathcal J)$ is called a {\em random measure}. A random measure $\mu$ is called {\em integer-valued} if it takes values in $\mathbb N\cup\{\infty\}$, i.e.\ for each $A \in \mathcal B(\mathbb R_+)\otimes \mathcal F\otimes \mathcal J$ one has that $\mu(A) \in \mathbb N\cup\{\infty\}$ a.s., and if $\mu(\{t\}\times J)\in \{0,1\}$ a.s.\ for all $t\geq 0$.  

Recall that $\mathcal P$ and $\mathcal O$ denote the predictable and optional $\sigma$-algebras on $\mathbb R_+\times \Omega$ and $\widetilde{\mathcal P}= \mathcal P \otimes \mathcal J$ and $\widetilde {\mathcal O}:=\mathcal O \otimes \mathcal J$ are the induced $\sigma$-algebras on $\widetilde {\Omega} := \mathbb R_+ \times \Omega \times J$. A~process $F:\mathbb R_+\times \Omega \to \mathbb R$ is called {\em optional} if it is $\mathcal O$-measurable. A random measure $\mu$ is called {\em optional} (resp. {\em predictable}) if for any $\widetilde{\mathcal O}$-measurable (resp. $\widetilde{\mathcal P}$-measurable) nonnegative $F:\mathbb R_+ \times \Omega \times J \to \mathbb R_+$ the stochastic integral 
\[
 (t,\omega) \mapsto \int_{\mathbb R_+\times J}\mathbf 1_{[0,t]}(s)F(s,\omega,x) \mu(\omega;\ud s, \ud x),\;\;\; t\geq 0, \ \omega \in \Omega,
\]
as a function from $\mathbb R_+ \times \Omega$ to $\overline{\mathbb R}_+$ is optional (resp.\ predictable).  

Let $X$ be a Banach space. 
Then we can extend stochastic integration to \mbox{$X$-valued} processes in the following way. Let $F:\mathbb R_+ \times \Omega \times J  \to X$, $\mu$ be a random measure. The integral
\[
 t\mapsto \int_{\mathbb R_+ \times J} F(s,\cdot, x)\mathbf 1_{[0,t]}(s)\mu(\cdot; \ud s, \ud x),\;\;\; t\geq 0,
\]
is well-defined and optional (resp.\ predictable) if $\mu$ is optional (resp.\ predictable), $F$ is $\widetilde {\mathcal O}$-strongly-measurable (resp.\ $\widetilde {\mathcal P}$-strongly-measurable), and $\int_{\mathbb R_+ \times J}\|F\|\ud\mu$ is a.s.\ bounded.

A random measure $\mu$ is called $\widetilde{\mathcal P}$-$\sigma$-finite if there exists an increasing sequence of sets $(A_n)_{n\geq 1}\subset \widetilde{\mathcal P}$ such that $\int_{\mathbb R_+ \times J} \mathbf 1_{A_n}(s,\omega,x)\mu(\omega; \ud s, \ud x)$ is finite a.s.\ and $\cup_n A_n = \mathbb R_+ \times \Omega \times J$. According to \cite[Theorem II.1.8]{JS} every $\widetilde{\mathcal P}$-$\sigma$-finite optional random measure $\mu$ has a {\em compensator}: a~unique $\widetilde{\mathcal P}$-$\sigma$-finite predictable random measure $\nu$ such that $\mathbb E \int_{\mathbb R_+ \times J}W \ud \mu = \mathbb E  \int_{\mathbb R_+ \times J}W \ud \nu$ for each $\widetilde{\mathcal P}$-measurable real-valued nonnegative $W$. We refer the reader to \cite[Chapter II.1]{JS} for more details on random measures. For any optional $\widetilde{\mathcal P}$-$\sigma$-finite measure $\mu$ we define the associated compensated random measure by $\bar{\mu} = \mu -\nu$. 

For each $\widetilde{\mathcal P}$-strongly-measurable $F:\mathbb R_+ \times \Omega \times J \to X$ such that 
$$
\mathbb E \int_{\mathbb R_+ \times J}\|F\| \ud \mu< \infty
$$ 
(or, equivalently, $\mathbb E \int_{\mathbb R_+ \times J}\|F\| \ud \nu<\infty$, see the definition of a compensator above) we can define a process $t \mapsto \int_{[0,t] \times J}\|F\| \ud  \bar{\mu}$ by $\int_{[0,t] \times J}F \ud  {\mu} - \int_{[0,t] \times J}F \ud  {\nu}$. The following lemma is a vector-valued version of \cite[Definition 1.27]{JS}.

\begin{lemma}\label{lem:intwrtrmidapdlocmart}
 Let $X$ be a Banach space, $\mu$ be a $\widetilde{\mathcal P}$-$\sigma$-finite optional random measure, $F:\mathbb R_+ \times \Omega \times J \to X$ be $\widetilde{\mathcal P}$-strongly-measurable such that $\mathbb E \int_{\mathbb R_+ \times J} \|F\|\ud \mu < \infty$. Then $\bigl(\int_{[0,t]\times J} F \ud \bar{\mu}\bigr)_{t\geq 0}$ is a purely discontinuous $X$-valued martingale.
\end{lemma}

\begin{proof}
 It is sufficient to show that 
 $$
 t\mapsto \Bigl\langle \int_{[0,t]\times J} F \ud \bar{\mu},x^* \Bigr\rangle =  \int_{[0,t]\times J} \langle F,x^*\rangle  \ud \bar{\mu},\;\;\; t\geq 0,
 $$
 is a purely discontinuous martingale for each $x^* \in X^*$, which can be shown similarly the discussion right below \cite[Definition 1.27]{JS}.
\end{proof}

The reader can find more information on random measures in \cite{JS,DY17,Kal,MarRo,Marinelli13,Nov75}.

\subsection{Predictable and totally inaccessible stopping times}\label{subsec:prandtotinacsttimes}

A stopping time $\tau$ is called {\em predictable} if there exists a sequence of stopping times $(\tau_n)_{n\geq 0}$ such that $\tau_n<\tau$ a.s.\ on $\{\tau > 0\}$ and $\tau_n \nearrow \tau$ a.s.\ as $n\to \infty$. A stopping time $\tau$ is called {\em totally inaccessible} if $\mathbb P(\tau=\sigma)=0$ for any predictable stopping time $\sigma$. Later we will need the following lemma.

\begin{lemma}\label{lem:DeltaV_tau=0forpredictVandtotinactau}
 Let $X$ be a Banach space, $V:\mathbb R_+ \times \Omega \to X$ be a predictable c\`adl\`ag process. Let $\tau$ be a totally inaccessible stopping time. Then $\Delta V_{\tau} = 0$ a.s.
\end{lemma}

\begin{proof}
 It is sufficient to show that $\langle \Delta V_{\tau}, x^*\rangle = 0$ a.s.\ for any $x^* \in X^*$. Then the statement follows from \cite[Proposition I.2.24]{JS}.
\end{proof}

Let $X$ be a Banach space, $M:\mathbb R_+ \times \Omega\to X$ be a local martingale. Then $M$ has a c\`adl\`ag version (see e.g.\ \cite{Y17FourUMD}), and therefore we can define adapted c\`adl\`ag process $M^{\tau-}=(M^{\tau-}_t)_{t\geq 0}$ in the following way
\begin{equation}\label{eq:defofM^tau-ingen}
 M^{\tau-}_t :=\lim_{\eps\to 0}M_{(\tau-\eps)\wedge t},\;\;\; t\geq 0,
\end{equation}
where we set $M_t=0$ for $t<0$. Notice that $M^{\tau-}$ is not necessarily a local martingale. For instance if $X = \mathbb R$ and $M$ is a compensated Poisson process, $\tau:= \inf_{t\geq 0}\{\Delta M_t>0\}$, then $M^{\tau-}_t = -(t\wedge \tau)$ a.s.\ for each $t\geq 0$, so it is a supermartingale which is not even a local martingale. Nevertheless, if $\tau$ is a predictable stopping time, then the following lemma holds. Recall that for any stopping time $\tau$ we define $\sigma$-field $\mathcal F_{\tau-}$ in the following way
$$
\mathcal F_{\tau-} := \sigma\{\mathcal F_0 \cup (\mathcal F_{t}\cap \{t<\tau\}), t> 0\}
$$
(see \cite[p.\ 491]{Kal} for details).

\begin{lemma}\label{lem:predtaupreservemartingales}
 Let $X$ be a Banach space, $M:\mathbb R_+ \times \Omega \to X$ be a local martingale, $\tau$ be a predictable stopping time. Then $M^{\tau-}$ defined as in \eqref{eq:defofM^tau-ingen} is a local martingale. Moreover, if $M$ is an $L^1$-martingale, then $M^{\tau-}$ is an $L^1$-martingale as well.
\end{lemma}

\begin{proof}
 Without loss of generality we can let $M_0=0$ a.s. First assume that $M$ is an $L^{\infty}$-martingale. Let $(\tau_n)_{n\geq 1}$ be an announcing to $\tau$ sequence of stopping times, i.e.\ $\tau_n <\tau$ a.s.\ on $\{\tau >0\}$ and $\tau_n\nearrow \tau$ a.s. as $n\to \infty$. Then $M^{\tau_n}$ is an $L^1$-martingale for each $n\geq 1$. Moreover, $M^{\tau_n}_t \to M^{\tau-}_t$ a.s.\ as $n\to \infty$ for each $t\geq 0$. On the other hand, $M^{\tau_n}_t = \mathbb E (M_t|\mathcal F_{\tau_n}) \to \mathbb E (M_t|\mathcal F_{\tau-})$ a.s.\ as $n\to \infty$ by \cite[Theorem 3.3.8]{HNVW1} and \cite[Lemma 25.2(iii)]{Kal}, and hence in $L^1$ by the uniform boundedness due to the boundedness of $M_{\infty}$. Therefore for each $t\geq 0$ we have that $M^{\tau-}_t =\mathbb E (M_t|\mathcal F_{\tau-})$ is integrable, hence for all $0\leq s\leq t$ 
 \[
  \mathbb E(M^{\tau-}_t|\mathcal F_{s}) = \mathbb E \bigl(\lim_{n\to \infty} M^{\tau_n}_t|\mathcal F_{s}\bigr) = \lim_{n\to \infty }\mathbb E (M^{\tau_n}_t|\mathcal F_{s}) = \lim_{n\to \infty}M^{\tau_n}_s = M^{\tau-}_s,
 \]
where all the limits are taken in $L^1(\Omega; X)$. Hence $(M^{\tau-}_t)_{t\geq 0}$ is a martingale. Moreover, by \cite[Corollary 2.6.30]{HNVW1}
\begin{equation}\label{eq:ineqforM^tau-_t}
 \mathbb E \|M^{\tau-}_t\| = \mathbb E \|\mathbb E (M_{t}|\mathcal F_{\tau-})\| \leq \mathbb E \|M_{t}\|\leq \mathbb E \|M_{\infty}\|,\;\;\; t\geq 0.
\end{equation}

\smallskip

Now we treat the general case.  Without loss of generality using a stopping time argument assume that $M$ is an $L^1$-martingale. Let $(M^m)_{m\geq 1}$ be a sequence of \mbox{$X$-va}\-lued $L^{\infty}$-martingales such that $M^m_{\infty} \to M_{\infty}$ in $L^1(\Omega; X)$ as $m\to \infty$. Analogously the first part of the proof $M^{\tau-}_t =\mathbb E (M_t|\mathcal F_{\tau-})$ for each $t\geq 0$; moreover, by \eqref{eq:ineqforM^tau-_t} $\bigl((M^m)^{\tau-}_t\bigr)_{m\geq 1}$ is a Cauchy sequence in $L^1(\Omega; X)$. Therefore by \cite[Corollary~2.6.30]{HNVW1}, $(M^m)^{\tau-}_t \to M^{\tau-}_t$ in $L^1(\Omega; X)$ for each $t\geq 0$, hence for each $t\geq s\geq 0$ by \cite[Corollary 2.6.30]{HNVW1}
\begin{align*}
 \mathbb E(M^{\tau-}_t|\mathcal F_{s}) = \mathbb E\bigl(\lim_{m\to \infty}(M^m)^{\tau-}_t|\mathcal F_{s}\bigr) &= \lim_{m\to \infty}\mathbb E((M^m)^{\tau-}_t|\mathcal F_{s})\\
 &= \lim_{m\to\infty}(M^m)^{\tau-}_s = M^{\tau-}_s,
\end{align*}
where all the limits are again taken in $L^1(\Omega; X)$. Therefore $(M^{\tau-}_t)_{t\geq 0}$ is an \mbox{$L^1$-mar}\-tin\-gale.
\end{proof}

\subsection{Compensator and variation}

Let $X$ be a Banach space, $M:\mathbb R_+ \times \Omega \to X$ be an adapted c\`adl\`ag process. Then a predictable process $V:\mathbb R_+ \times \Omega$ is called a {\em predictable	compensator} of $M$ (or just a {\em compensator} of $M$) if $V_0=0$ a.s.\ and if $M-V$ is a local martingale.

\smallskip

The {\em variation} $\Var M:\mathbb R_+ \times \Omega \to \overline{\mathbb R}_+$ of a c\`adl\`ag process $M:\mathbb R_+ \times \Omega\to X$ is defined in the following way:
\begin{equation}\label{eq:defofVarofM}
  (\Var M)_t:=\|M_0\| + \limsup_{{\rm mesh}\to 0}\sum_{n=1}^N \|M(t_n)-M(t_{n-1})\|,
\end{equation}
where the limit superior is taken over all the partitions $0= t_0 < \ldots < t_N = t$.

\smallskip

Let $V:\mathbb R_+ \times \Omega\to X$ be a c\`adl\`ag adapted process. Analogously to the scalar-valued situation we can define a c\`adl\`ag adapted process $V^*:\mathbb R_+ \times \Omega \to \mathbb R_+$ of the following form
\[
 V^*_t := \sup_{s\in[0,t]}\|V_s\|,\;\;\; t\geq 0.
\]

\section{Gundy's decomposition of continuous-time martingales}\label{sec:Gundy'sdec}

For the proof of our main results, Theorem \ref{thm:candecoflocmartsufandnesofUMD} and Theorem \ref{thm:analougeofProp3.5.4}, we will need Gundy's decomposition of continuous-time martingales, which is a generalization of Gundy's decomposition of discrete martingales (see \cite{Gundy68} and \cite[Theorem~3.4.1]{HNVW1} for the details).

\begin{theorem}[Gundy's decomposition]\label{thm:Gundyconttime}
 Let $X$ be a Banach space, $M:\mathbb R_+ \times \Omega \to X$ be a martingale. Then for each $\lambda >0$ there exist martingales $M^1$, $M^2$ $M^3:\mathbb R_+ \times \Omega \to X$ such that $M = M^1 + M^2 + M^3$ and
 \begin{itemize}
  \item [(i)] $\|M^1_{t}\|_{L^{\infty}(\Omega; X)}\leq 2\lambda$, $\mathbb E\|M_t^1\|\leq 5 \mathbb E\|M_t\|$ for each $t\geq 0$,
  \item [(ii)] $\lambda \mathbb P ((M^2)^*_t>0)\leq 4\mathbb E\|M_t\|$ for each $t\geq 0$,
  \item [(iii)] $\mathbb E(\Var M^3)_t\leq 7 \mathbb E\|M_t\|$ for each $t\geq 0$.
 \end{itemize}
\end{theorem}

\begin{remark}
 Notice that if $M$ is a discrete martingale (i.e.\ $M_t = M_{[t]}$ for any $t\geq 0$), then the decomposition in Theorem \ref{thm:Gundyconttime} turns to the classical discrete one from \cite[Theorem~3.4.1]{HNVW1}.
\end{remark}

For the proof we will need the following intermediate steps.

\begin{lemma}\label{lem:goodcompensator}
 Let $X$ be a Banach space, $M:\mathbb R_+ \times \Omega \to X$ be a c\`adl\`ag adapted process such that $\mathbb E (\Var M)_t <\infty$ for each $t\geq 0$ and a.s.
 \[
  M_t = \sum_{0\leq s\leq t} \Delta M_s,\;\;\; t\geq 0.
 \]
Then $M$ has a c\`adl\`ag predictable compensator $V:\mathbb R_+ \times \Omega \to X$ such that for each $t\geq 0$
 \begin{equation}\label{eq:lemgoodcompensator1}
  \mathbb E \|V_t\|\leq \mathbb E (\Var V)_t \leq\mathbb E (\Var M)_t.
 \end{equation}
 In particular, if $M$ has a.s.\ at most one jump, then
 \begin{equation}\label{eq:lemgoodcompensator2}
  \mathbb E \|V_t\|\leq \mathbb E (\Var V)_t \leq \mathbb E (\Var M)_t = \mathbb E \|M_t\|.
 \end{equation}
\end{lemma}
\begin{proof}
 Let $\mu^M$ be a random measure defined on $\mathbb R_+\times X$ pointwise in $\omega\in \Omega$ in the following way:
 \begin{equation}\label{eq:defofranmeasuremu^M}
    \mu^M(\omega; B\times A) := \sum_{u\in B} \mathbf 1_{A\setminus\{0\}}(\Delta M_u(\omega)),\;\;\; \omega\in \Omega, B\in \mathcal B(\mathbb R_+), A \in \mathcal B(X).
 \end{equation}

Notice that $(\Var M)_t = \sum_{0\leq s\leq t} \|\Delta M_s\|$ a.s.\ for each $t\geq 0$, so in particular a.s.\
\begin{equation}\label{eq:VarMequalsintwrtranmeas}
 (\Var M)_t = \int_{[0,t]\times X}\|x\|\ud \mu^M(x,s),\;\;\;t\geq 0.
\end{equation}

Also note that $\mu^M$ is $\mathcal P$-$\sigma$-finite: for each $0\leq u\leq v$ and $t\geq 0$ one has that
\begin{align*}
  \mathbb E \int_{[0,t]\times X} \mathbf 1_{\|x\|\in [u,v]}\ud \mu^M &\eqsim_{u,v}\mathbb E \int_{[0,t]\times X} \|x\|\mathbf 1_{\|x\|\in [u,v]}\ud \mu^M\\
  &\leq \mathbb E \int_{[0,t]\times X} \|x\|\ud \mu^M \\
  &= \mathbb E(\Var M)_t <\infty.
\end{align*}

Since $\mu^M$ is an integer-valued optional $\mathcal P$-$\sigma$-finite random measure, it has a predictable compensator $\nu^M$ (see Subsection \ref{subsec:randommeasures} and \cite[Theorem II.1.8]{JS}), and therefore since by \eqref{eq:VarMequalsintwrtranmeas}
\[
 \mathbb E \int_{[0,t]\times X}\|x\|\ud \mu^M(x,s) =\mathbb E (\Var M)_t < \infty,
\]
we have that
 $$
 t\mapsto V_t := \int_{[0,t]\times X}x \ud \nu^M(x,s), \;\;\; t\geq 0,
 $$ 
 is integrable and c\`adl\`ag in time due to the fact that it is an integral with respect to the measure $\nu^M$ a.s. Moreover, by the definition of variation \eqref{eq:defofVarofM} we have that $\|V_t\|\leq (\Var V)_t$ a.s.\ for each $t\geq 0$, and hence
 \begin{align*}
  \mathbb E \|V_t\| \leq \mathbb E (\Var V)_t  \leq \mathbb E \int_{[0,t]\times X}\|x\| \ud \nu^M(x,s)&\stackrel{(*)}=\mathbb E \int_{[0,t]\times X}\|x\| \ud \mu^M(x,s)\\
  &\stackrel{(**)} = \mathbb E (\Var M)_t,
 \end{align*}
 where $(*)$ holds due to the definition of a compensator, and $(**)$ follows from \eqref{eq:VarMequalsintwrtranmeas}. To show \eqref{eq:lemgoodcompensator2} it is sufficient to notice that if $M$ has at most one jump then $(\Var M)_t =  \|M_t\|$ a.s.\ for each $t\geq 0$.
\end{proof}

The following lemma is folklore, but the author could not find an appropriate reference, so we present it with the proof here.

\begin{lemma}\label{lemma:locbddnessofpredicproc}
 Let $X$ be a Banach space, $V:\mathbb R_+ \times \Omega \to X$ be a right-continuous predictable process, $V_0=0$ a.s. Then $V$ is locally bounded.
\end{lemma}

\begin{proof}
 For each $n\geq 0$ define a stopping time $\tau_n:= \inf\{t\geq 0: \|V_t\|\geq n\}$. Then a sequence $(\tau_n)_{n\geq 1}$ of stopping times is increasing a.s.\ and tends to infinity as $n\to \infty$. Moreover, $(\tau_n)_{n\geq 1}$ are predictable by \cite[Theorem 25.14]{Kal} and the fact that for each $n\geq 1$
 \begin{equation}\label{eq:tauispredictifVispredict}
  \{\tau\leq t\} = \{ \sup_{0\leq s\leq t}\|V_s\|\geq n \}\in \mathcal P.
 \end{equation}
Therefore for each $n\geq 1$ there exists an announcing sequence $(\tau_{m,n})_{m\geq 1}$ of stopping times. Choose $m_n$ so that $\mathbb P(\tau_n - \tau_{m_n,n}>\frac{1}{2^n})<\frac{1}{2^n}$. Then $(\tau_{m_n,n})_{n\geq 1}$ is such that $\tau_{m_n,n}\to \infty$ a.s.\ as $n\to \infty$, and for each $n\geq 0$ we have that a.s.\ $\sup_{0\leq s\leq \tau_{m_n,n}}\|V_s\| \leq \sup_{0\leq s< \tau_{n}}\|V_s\| \leq n$.
\end{proof}

Let $\tau$ and $\sigma$ be stopping times. Then we can set
\begin{equation}\label{eq:tau-wedgesigma-:=(tauwedgesigma)-}
 \tau - \wedge \sigma- := (\tau \wedge \sigma)-.
\end{equation}
Notice that if $M:\mathbb R_+\times \Omega \to X$ is a c\`adl\`ag process, then $(M^{\tau-})^{\sigma-} = M^{\tau- \wedge \sigma-}$.

\begin{proof}[Proof of Theorem \ref{thm:Gundyconttime}]
 By a stopping time argument we can assume that $M$ is an \mbox{$L^1$-mar}\-tin\-gale. Define a stopping time $\tau$ is the following way: 
 $$
 \tau=\inf\Bigl\{t\geq 0: \|M_t\|\geq \frac{\lambda}{2}\Bigr\}.
 $$
 Let $M^{2,1}: = M-M^{\tau}$ and let $M^{3,1}(\cdot) = \Delta M_{\tau}\mathbf 1_{[0,\cdot]}(\tau) + M_0^{\tau-}$, where by \eqref{eq:defofM^tau-ingen} we can conclude that a.s.\
\begin{equation}\label{eq:defofM_0^tau-}
  M_0^{\tau-} :=
\begin{cases}
 M_0,\;\;\; &\tau>0,\\
 0,\;\;\;&\tau=0.
\end{cases}
\end{equation}

Let $N:\mathbb R_+ \times \Omega \to X$ be such that $N_t = \Delta M_{\tau}\mathbf 1_{[0,t]}(\tau)$, $t\geq 0$. Then due to the fact that $M_{\tau}=\mathbb E (M_{\infty}|\mathcal F_{\tau})$ by \cite[Theorem 7.29]{Kal}, \cite[Corollary 2.6.30]{HNVW1}, and the fact that $\|M_{\tau-}\|\leq \frac{\lambda}{2}$ a.s., we get
\begin{equation}\label{eq:ineqforvarN}
 \begin{split}
  \mathbb E (\Var N)_{\infty}=\mathbb E \|\Delta M_{\tau}\| = \mathbb E \|M_{\tau} - M_{\tau-}\|&\leq \mathbb E \|M_{\tau}\| + \mathbb E (\|M_{\tau-}\|\mathbf 1_{\tau<\infty})\\
 &\leq \mathbb E \|M_{\infty}\| + \frac{\lambda}{2} <\infty.
 \end{split}
\end{equation}
Therefore by Lemma \ref{lem:goodcompensator}, $N$ has a compensator $V$. Let 
$$
\sigma:= \inf\{t\geq 0:\|V_t\|\geq \lambda\}
$$ 
be a stopping time. Then by \eqref{eq:tauispredictifVispredict} $\sigma$ is a predictable stopping time. Define now $M^1 = M^{\sigma-\wedge\tau-} + V^{\sigma-} - M_0^{\tau-}$, $M^{2,2} = (M^{\tau-} + V)-(M^{\sigma-\wedge\tau-} + V^{\sigma-})$, $M^{3,2} = N-V$ where $\sigma- \wedge \tau-$ is defined as in \eqref{eq:tau-wedgesigma-:=(tauwedgesigma)-}. Define $M^2 := M^{2,1} + M^{2,2}$ and $M^3 := M^{3,1} + M^{3,2}$. Then $M = M^1 + M^2 + M^3$. Now let us describe why this is the right choice.

{\em Step 1: $M^1$.} First show that $M^1$ is a martingale. Indeed, for each $t\geq 0$
\begin{equation}\label{eq:reprofM^1}
 \begin{split}
   M^1_t &= M^{\sigma-\wedge \tau-}_t + V^{\sigma-}_t - M_0^{\tau-} = (M^{\tau-}_t + V_t-M_0^{\tau-})^{\sigma-}\\
 &=\bigl( M^{\tau}_t - \mathbf 1_{\tau\in [0,t]} \Delta M_{\tau} + V_t -M_0^{\tau_2-}\bigr)^{\sigma-}\\
 &= \Bigl( (M^{\tau}_t-M_0^{\tau-}) - (N_t-V_t) \Bigr)^{\sigma-},
 \end{split}
\end{equation}
and the last expression is a martingale due to the fact that $M^{\tau}$ is a martingale by \cite[Theorem 7.12]{Kal}, 
the fact that $N-V$ is a martingale by the definition of a~compensator, Lemma \ref{lem:predtaupreservemartingales}, and the fact that by \eqref{eq:ineqforvarN}
$$
\mathbb E \|N_{\infty}\| \leq \mathbb E (\Var N)_{\infty}\leq \mathbb E \|M_{\infty}\| +\frac{\lambda}2 <\infty.
$$

Now let us check (i): $\|M_{\infty}^{\sigma-\wedge\tau-}\|$, $\|M_0^{\tau-}\|\leq \frac {\lambda}{2}$ a.s.\ by the definition of $\tau$, and $\|V_{\infty}^{\sigma-}\|\leq \lambda$ by the definition of $\sigma$, so $\|M^1_{\infty}\|\leq 2\lambda$ a.s.

Further, to prove the second part of (i) we will use the representation of $M^1$ from the last line of \eqref{eq:reprofM^1}. Notice that by \cite[Theorem 7.12]{Kal} and \cite[Corollary 2.6.30]{HNVW1} for each fixed $t\geq 0$
\begin{equation}\label{eq:someofpredictableguysM^tau_1}
 \mathbb E \|M^{\tau}_t\| \leq \mathbb E \|M_t\|.
\end{equation}
Moreover,
\begin{multline*}
  \mathbb E \|N_t\| = \mathbb E \|M^{\tau}_t - M^{\tau-}_t\|\leq \mathbb E \|M^{\tau}_t\| +\mathbb E(\| M^{\tau-}_t\|\mathbf 1_{\tau<\infty})\\
  \leq \mathbb E \|M_t^{\tau}\| + \mathbb E\Bigl(\frac{\lambda}{2}\mathbf 1_{\tau <\infty} \Bigr) \leq 2\mathbb E \|M_t^{\tau}\| \stackrel{(*)}\leq 2\mathbb E\|M_t\|,
\end{multline*}
where $\| M^{\tau-}_t\|\leq \frac{\lambda}{2} \leq \| M^{\tau}_t\|$ on $\{\tau <\infty\}$ by the definition of $\tau$, and $(*)$ follows from \cite[Theorem 7.12]{Kal} and \cite[Corollary 2.6.30]{HNVW1}. Therefore by \eqref{eq:lemgoodcompensator2}
\begin{equation}\label{eq:estforV_tby2M_t}
 \mathbb E \|V_t\|\leq \mathbb E \|N_t\|\leq 2\mathbb E\|M_t\|
\end{equation}
as well. Finally, $\mathbb E \|M_0^{\tau-}\| \leq \mathbb E \|M_0\|\leq \mathbb E \|M_t\|$ by \eqref{eq:defofM_0^tau-} and \cite[Corollary 2.6.30]{HNVW1}. Consequently, the second part of (i) holds by the estimates above and by the triangle inequality.

{\em Step 2: $M^2$.} First note that
\begin{equation}\label{eq:formulaforM^2}
  M^2 = M- M^{\tau} + (M^{\tau-} + V)-(M^{\tau-} + V)^{\sigma-}.
\end{equation}
Let us check that $M^2$ is a martingale. $M-M^{\tau}$ is a martingale by \cite[Theorem~7.12]{Kal}. Furthermore,
\[
 M^{\tau-} + V = M^{\tau}  -(N-V)
\]
is a martingale as well due to \cite[Theorem 7.12]{Kal} and the fact that $V$ is a compensator of $N$. Finally, $(M^{\tau-} + V)^{\sigma-}$ is a martingale by Lemma~\ref{lem:predtaupreservemartingales}.

Let us now prove (ii). Notice that by \eqref{eq:formulaforM^2}
\[
 \mathbb P((M^2)^*_t>0) \leq \mathbb P((M- M^{\tau})^*_t>0) + \mathbb P(((M^{\tau-} + V)-(M^{\tau-} + V)^{\sigma-})^*_t>0).
\]
First estimate $\mathbb P((M- M^{\tau})^*_t>0)$:
\begin{align*}
 \mathbb P((M- M^{\tau})^*_t>0) \leq \mathbb P(\tau\leq t) \leq \mathbb P\Bigl(M^*_t\geq \frac{\lambda} 2\Bigr)\leq \frac{2\mathbb E\|M_t\|}{\lambda},
\end{align*}
where the latter inequality holds by \eqref{eq:1,inftyineqforsubmart}. Using the same machinery we get
\begin{multline*}
 \mathbb P(((M^{\tau-} + V)-(M^{\tau-} + V)^{\sigma-})^*_t>0) \leq \mathbb P(\sigma \leq t) \\
 = \mathbb P(\|V_t\|\geq \lambda) \stackrel{(i)}\leq \frac{\mathbb E\|V_t\|}{\lambda} \stackrel{(ii)}\leq \frac{2\mathbb E\|M_t\|}{\lambda},
\end{multline*}
where $(i)$ follows from the Chebyshev inequality, and $(ii)$ follows from \eqref{eq:estforV_tby2M_t}. This terminates the proof of (ii).

{\em Step 3: $M^3$.} Recall that
\[
 M^3 =  M_0^{\tau-} + N-V.
\]
Therefore by the triangle inequality a.s.\ for each $t\geq 0$
\begin{equation}\label{eq:estforVarofM^3}
 \begin{split}
  \mathbb E (\Var M^3)_t &\leq \mathbb E\|M_0^{\tau-}\| + \mathbb E(\Var N)_t + \mathbb E(\Var V)_t\\
  &\leq \mathbb E \|M_t\| + 2\mathbb E \|N_t\| \leq 5\mathbb E\|M_t\|,
 \end{split}
\end{equation}
where the latter inequality holds by \eqref{eq:estforV_tby2M_t}, while the rest follows from 
\eqref{eq:lemgoodcompensator1}
and the fact that $\mathbb E\|M_0^{\tau-}\| \leq \mathbb E\|M_0\| \leq \mathbb E\|M_t\|$.
\end{proof}

\begin{remark}\label{rem:localLpintegrabilityGundysdec}
 Let $p\in (1,\infty)$, $M$ be an $L^p$-martingale, $\lambda >0$, $M=M^1 + M^2 + M^3$ be Gundy's decomposition (see the theorem above). Then $M^1$ is an $L^p$ martingale since $\|M^1_{t}\|_{L^{\infty}(\Omega; X)}\leq 2\lambda$ for all $t\geq 0$; $M^3$ is a local $L^p$-martingale since $M^3 = M_0^{\tau-} + N-V$, where both $M_0^{\tau-}$ and $N_{\infty} = \Delta M_{\tau}$ are $L^p$-integrable (the latter is $L^p$-integrable by the argument similar to \eqref{eq:ineqforvarN}), and $V$ is locally $L^p$-integrable by Lemma \ref{lemma:locbddnessofpredicproc}; finally, $M^2$ is a local $L^p$-martingale since $M^2 = M - M^1 - M^3$. Therefore all the martingales in Gundy's decomposition are locally $L^p$-integrable given $M$ is an $L^p$-martingale.
\end{remark}

\section{The canonical decomposition of local martingales}\label{sec:candecoflocalmartUMD}

The current section is devoted to the proof of the fact that the canonical decomposition (as well as the Meyer-Yoeurp and the Yoeurp decompositions) of any $X$-valued local martingale exists if and only if $X$ has the UMD Banach property. Recall that the Meyer-Yoeurp decomposition split a local martingale $M$ into a sum $M=M^c + M^d$ of a continuous local martingale $M^c$ and a purely discontinuous local martingale $M^d$, while the Yoeurp decomposition split a purely discontinuous local martingale $M^d$ into a sum $M^d=M^q + M^a$ of a quasi-left continuous local martingale $M^q$ and a local martingale $M^a$ with accessible jumps.

First we give all the basic definitions properly, and thereafter we provide the reader with the proof of the main statement, Theorem \ref{thm:candecoflocmartsufandnesofUMD}.

\subsection{Basic definitions and decompositions of $L^p$-martingales}\label{subsec:decompositionsofmartingales}

Let $X$ be a~Banach space. Recall that a purely discontinuous local martingale have been defined in Subsection \ref{subsec:pdmart}.
\begin{definition}\label{def:MYdecoflocmart}
 A local martingale $M:\mathbb R_+ \times \Omega \to X$ is called to have the {\em Meyer-Yoeurp decomposition} if there exist local martingales $M^c, M^d:\mathbb R_+\times \Omega \to X$ such that $M^c$ is continuous, $M^d$ is purely discontinuous, $M^c_0=0$, and $M = M^c + M^d$.
\end{definition}

\begin{remark}\label{rem:MY==weakMY}
 Recall that by \cite{Y17MartDec} if $M = M^c + M^d$ is the Meyer-Yoeurp decomposition, then $\langle M^c, x^*\rangle$ is continuous and $\langle M^d, x^*\rangle$ is purely discontinuous for any $x^* \in X^*$; therefore this decomposition is unique by the uniqueness of the Meyer-Yoeurp decomposition of a real-valued local martingale (see \cite[Theorem 26.14]{Kal} and \cite{Y17MartDec} for details).
\end{remark}

Let $M:\mathbb R_+ \times \Omega \to X$ be a local martingale. Then $M$ is called {\em quasi-left continuous} if $\Delta M_{\tau}=0$ a.s.\ for any predictable stopping time $\tau$, and $M$ is called {\em with accessible jumps} if $\Delta M_{\tau}=0$ a.s.\ for any totally inaccessible stopping time~$\tau$ (see Subsection \ref{subsec:prandtotinacsttimes} for the definition of a predictable and a totally inaccessible stopping times).

\begin{definition}\label{def:Ydecofpdlocmart}
 A purely discontinuous local martingale $M^d:\mathbb R_+ \times \Omega \to X$ is called to have the {\em Yoeurp decomposition} if there exist purely discontinuous local martingales $M^q, M^a:\mathbb R_+\times \Omega \to X$ such that $M^q$ is quasi-left continuous, $M^a$ has accessible jumps, $M^q_0=0$, and $M^d = M^q + M^a$.
\end{definition}

\begin{remark}\label{rem:Y==weakY}
Analogously to Remark \ref{rem:MY==weakMY} it follows from \cite[Corollary 26.16]{Kal} that the Yoeurp decomposition is unique.
\end{remark}

Composing Definition \ref{def:MYdecoflocmart} and \ref{def:Ydecofpdlocmart} we get the canonical decomposition.

\begin{definition}
 A local martingale $M:\mathbb R_+ \times \Omega \to X$ is called to have the {\em canonical decomposition} if there exist local martingales $M^c,M^q, M^a:\mathbb R_+\times \Omega \to X$ such that $M^c$ is continuous, $M^q$ and $M^a$ are purely discontinuous, $M^q$ is quasi-left continuous, $M^a$ has accessible jumps, $M^c_0=M^q_0=0$, and $M = M^c +  M^q + M^a$.
\end{definition}

\begin{remark}\label{rem:decompsitionspreserveDeltasRvaluedcase}
 Notice that if $M = M^c+ M^q + M^a$ is the canonical decomposition, then $\Delta M^q_{\tau} = \Delta M_{\tau}$ for any totally inaccessible stopping time $\tau$ since in this case $\Delta M^c_{\tau} = \Delta M^a_{\tau} = 0$ by the definition of a continuous local martingale and a local martingale with accessible jumps. Analogously, $\Delta M^a_{\tau} = \Delta M_{\tau}$ for any predictable stopping time $\tau$.
\end{remark}

The reader can find further details on the martingale decomposition discussed above in \cite{Kal,JS,Y17MartDec,DY17,Yoe76,Mey76}.

\smallskip

Due to \cite{Y17MartDec} the UMD property guarantees the canonical decomposition of any $X$-valued $L^p$-martingale with $p\in (1,\infty)$ and the following proposition holds:
\begin{proposition}\label{prop:candecofUMDspacevalL^pmart}
 Let $X$ be a UMD Banach space, $p\in (1,\infty)$. Then any \mbox{$L^p$-mar}\-tin\-gale $M:\mathbb R_+ \times \Omega \to X$ has the canonical decomposition $M = M^c +  M^q + M^a$, and then for each $t\geq 0$ we have that
 \begin{equation}\label{eq:L^pestforprop:candecofUMDspacevalL^pmart}
  \begin{split}
   \mathbb E \|M^c_t\|^p \leq \beta_{p,X}^p \mathbb E \|M_t\|^p,\\
   \mathbb E \|M^q_t\|^p \leq \beta_{p,X}^p \mathbb E \|M_t\|^p,\\
   \mathbb E \|M^a_t\|^p \leq \beta_{p,X}^p \mathbb E \|M_t\|^p,
  \end{split}
 \end{equation}
 where $\beta_{p, X}$ is the UMD$_p$ constant of $X$.
\end{proposition}

It is a natural question whether the canonical decomposition is possible and whether one can extend \eqref{eq:L^pestforprop:candecofUMDspacevalL^pmart} in the case $p=1$.
It turns out that the UMD property is necessary and sufficient for the canonical decomposition of a general local martingale, while instead of \eqref{eq:L^pestforprop:candecofUMDspacevalL^pmart} one gets weak-type estimates:

\begin{theorem}[Canonical decomposition of local martingales]\label{thm:candecoflocmartsufandnesofUMD}
 Let $X$ be a Banach space. Then $X$ has the UMD property if and only if any local martingale $M:\mathbb R_+\times \Omega \to X$ has the canonical decomposition $M = M^c +  M^q + M^a$. If this is the case, then for any $\lambda >0$ and $t\geq 0$
 \begin{equation}\label{eq:L^1inftyforcor:candecoflocmartsufofUMD}
  \begin{split}
  \lambda \mathbb P ((M^c)^*_t >\lambda) &\lesssim_X \mathbb E \|M_t\|,\\
   \lambda \mathbb P ((M^q)^*_t >\lambda) &\lesssim_X \mathbb E \|M_t\|,\\
    \lambda \mathbb P ((M^a)^*_t >\lambda) &\lesssim_X \mathbb E \|M_t\|.
  \end{split}
 \end{equation}
\end{theorem}

For the proof of the main theorem we will need a considerable amount of machinery, which will be provided in Subsection \ref{subsec:WDSoperator}-\ref{subsec:NecofUMD}.

\subsection{Weak differential subordination martingale transforms}\label{subsec:WDSoperator}

The current subsection is devoted to the proof of the fact that boundedness of a continuous-time martingale transform from a certain specific class acting on $L^p$-martingales implies the corresponding weak $L^1$-estimates. Such type of assertions for special discrete martingale transforms was first obtained by Burkholder in~\cite{Burk66}. Later the Burkholder's original statement was widely generalized in different directions (see \cite{Burk81,Hit90,HNVW1,MT00,BG70,GW05}), even though the martingale transforms were remaining acting on discrete martingales. The propose of the current section is to provide new results for martingale transforms of the same spirit by considering continuous-time martingales. This will allow us to consider linear operators that map a local martingale to the continuous part of the canonical decomposition, or the part of the canonical decomposition which is purely discontinuous with accessible jumps, so weak $L^1$-estimates \eqref{eq:L^1inftyforcor:candecoflocmartsufofUMD} will follow from $L^p$-estimates \eqref{eq:L^pestforprop:candecofUMDspacevalL^pmart} and Theorem \ref{thm:analougeofProp3.5.4}.

\smallskip

Before proving the main statement (Theorem \ref{thm:analougeofProp3.5.4}) we need to provide the reader with basic definitions. Let $M:\mathbb R_+ \times \Omega \to \mathbb R$ be a local martingale. We define a {\em quadratic 
variation} of $M$ in the following way:
\begin{equation}\label{eq:defquadvar}
  [M]_t  := \mathbb P-\lim_{{\rm mesh}\to 0}\sum_{n=1}^N |M(t_n)-M(t_{n-1})|^2,
\end{equation}
where the limit in probability is taken over partitions $0= t_0 < \ldots < t_N = 
t$. The reader can find more about a quadratic variation in \cite{Kal,Prot,MP,JS}.

Let $M,N:\mathbb R_+ \times \Omega \to \mathbb R$ be local martingales. Then $N$ is called to be {\em differentially subordinated} to $M$ (or $N\ll M$) if $|N_0|\leq|M_0|$ a.s.\ and $[N]_t-[N]_s\leq [M]_t-[M]_s$ a.s.\ for each $0\leq s\leq t<\infty$. We recommend the reader \cite{HNVW1,Os12,Wang,BB,BBB,Burk84} for further acquaintance with differential subordination.

Let $X$ be a Banach space, $M, N:\mathbb R_+ \times \Omega \to X$ be local martingales. Then $N$ is called to be {\em weakly differentially subordinated} to $M$ (or $N \stackrel{w}\ll M$) if $\langle N, x^*\rangle$ is differentially subordinated to $\langle M, x^*\rangle$ for each $x^* \in X^*$. The reader can find more details on weak differential subordination in \cite{Y17FourUMD,Y17UMD^A,Y17MartDec,OY18}.

The following theorem will be an important tool to show Theorem \ref{thm:candecoflocmartsufandnesofUMD} and it is connected with \cite[Proposition 3.5.4]{HNVW1}. Recall that $\mathcal M^{p}_X $ is a space of all \mbox{$L^p$-in}\-teg\-rable $X$-valued martingales, and $\mathcal M^{p, \rm loc}_X $ is a space of all locally $L^p$-integrable $X$-valued martingales (see Subsection \ref{subsec:martandcadlagprocesses}).

\begin{theorem}\label{thm:analougeofProp3.5.4}
 Let $X$ be a Banach space, $p\in (1,\infty)$, $T:\mathcal M^{p, \rm loc}_X \to \mathcal M^{p, \rm loc}_X$ be a~linear operator such that $TM \stackrel{w}\ll M$ and
 \begin{equation}\label{eq:starconditionforthm:analougeofProp3.5.4}
  M^*_{\infty} = 0\Longrightarrow (TM)^*_{\infty} = 0\;\;\; \text{a.s.}
 \end{equation}
for each $M \in \mathcal M^{p}_X$.
Assume that $T \in \mathcal L(\mathcal M^{p}_X)$. Then 
for any $M\in \mathcal M^p_X$
\begin{equation}\label{eq:estforL^1L^1inftynormthrL^pnorm}
 \lambda\mathbb P(\|(TM)^*_{\infty}\|>\lambda) \leq C_{p, T, X} \mathbb E \|M_{\infty}\|,\;\;\; \lambda >0,
\end{equation}
where $C_{p,T, X} = 26\|T\|_{\mathcal L(\mathcal M^p_X)}\frac{p}{p-1}+28$.
\end{theorem}

\begin{remark}
 Notice that if $X$ is a UMD Banach space, then $T$ is automatically bounded on $\mathcal M^{p}_X$ and $\|T\|_{\mathcal L(\mathcal M^p_X)} \leq \beta_{p,X}^2(\beta_{p, X}+1)$ by \eqref{eq:introL^pestWDS} and \cite{Y17MartDec} since $TM \stackrel{w}\ll M$ for any $M\in \mathcal M^{p}_X$.
\end{remark}

For the proof we will need several lemmas.

\begin{lemma}\label{lem:M=intwrtcompranmeas}
 Let $X$ be a Banach space, $M:\mathbb R_+ \times \Omega \to X$ be a purely discontinuous martingale with $M_0=0$ a.s. Let $\mu^M$ be the corresponding random measure defined as in \eqref{eq:defofranmeasuremu^M}. Assume that
 \begin{equation}\label{eq:lemmaintwrtrmequalspdmartXvaluedineqforjumps}
  \mathbb E \sum_{s\geq 0} \|\Delta M_s\| = \mathbb E\int_{\mathbb R_+ \times X} \|x\|\ud \mu^M<\infty.
 \end{equation}
Then $M_t = \int_{[0,t]\times X}x\ud \bar{\mu}^M$ for each $t\geq 0$ a.s.
\end{lemma}

\begin{proof}
 By \eqref{eq:lemmaintwrtrmequalspdmartXvaluedineqforjumps} there exists $N:\mathbb R_+ \times \Omega \to X$ such that $N_t = \sum_{0\leq s\leq t}\Delta M_s$ for each $t\geq 0$. Let $V = N-M$.
 Then both $t\mapsto N_t-V_t = M_t$, $t\geq 0$, and 
 $$
 t\mapsto N_t-\int_{[0,t]\times X}x\ud \nu^M = \int_{[0,t]\times X}x\ud \mu^M - \int_{[0,t]\times X}x\ud \nu^M = \int_{[0,t]\times X}x\ud \bar{\mu}^M,\;\; t\geq 0,
 $$
 are martingales. 
 Therefore
 \[
  t\mapsto V_t - \int_{[0,t]\times X}x\ud \nu^M = M_t-\int_{[0,t]\times X} x\ud \bar{\mu}^M,\;\;\; t\geq 0,
 \]
is a predictable martingale, which is purely discontinuous as a difference of two purely discontinuous martingales (see Lemma \ref{lem:intwrtrmidapdlocmart}). On the other hand it is continuous by the predictability (see e.g.\ \cite[Theorem 4]{Kr16} and \cite[Corollary 2.1.42]{HOSAAA2002}). Hence by Lemma \ref{lemma:contpuredisczero} this martingale equals zero since it starts at zero, so $M = N-V = \int_{[0,\cdot] \times X}x\ud \bar{\mu}^M$.
 \end{proof}

\begin{lemma}\label{lem:WDSvariations}
 Let $X$ be a Banach space, $M, N:\mathbb R_+ \times \Omega \to X$ be purely discontinuous martingales such that $N\stackrel{w}\ll M$. Then $\mathbb E (\Var N)_t \leq  2 \mathbb E (\Var M)_t$ for each $t\geq 0$.
\end{lemma}

\begin{proof}
 Without loss of generality $\mathbb E (\Var M)_{\infty}<\infty$. Notice that since $N\stackrel{w}\ll M$, for a.e.\ $(t,\omega)\in \mathbb R_+ \times \Omega$ there exists $a(t,\omega)\in [-1,1]$ such that $\Delta N_t(\omega) = a(t,\omega)\Delta M_t(\omega)$ (see \cite{Y17FourUMD}). Therefore a.s.\ for each $t\geq 0$
 \begin{equation}\label{eq:proofoflemmaonvariationsWDSpdmart}
  \begin{split}
     \int_{[0,t]\times X} \|x\|\ud \mu^N(x,s) &=\sum_{0\leq s\leq t} \|\Delta N_s\| = \sum_{0\leq s\leq t} |a(s,\cdot)| \|\Delta M_s\|\\
  &\leq \sum_{0\leq s\leq t} \|\Delta M_s\| \leq (\Var M)_t.
  \end{split}
 \end{equation}
So by Lemma \ref{lem:M=intwrtcompranmeas} $N = \int_{[0,\cdot]\times X}x\ud \bar{\mu}^N$, hence
\begin{align*}
 (\Var N)_t &= \Bigl(\Var \int_{[0,\cdot]\times X} x\ud \bar{\mu}^N(x,s)\Bigr)_t \\
 &= \Bigl(\Var \Bigl( \int_{[0,\cdot]\times X} x\ud {\mu}^N(x,s) -\int_{[0,\cdot]\times X} x\ud {\nu}^N(x,s) \Bigr)\Bigr)_t \\
 &\leq \Bigl(\Var \int_{[0,\cdot]\times X} x\ud {\mu}^N(x,s)\Bigr)_t + \Bigl(\Var \int_{[0,\cdot]\times X} x\ud {\nu}^N(x,s)\Bigr)_t\\
 &\leq \int_{[0,t]\times X} \|x\|\ud {\mu}^N(x,s) + \int_{[0,t]\times X} \|x\|\ud {\nu}^N(x,s)\\
 &= 2\int_{[0,t]\times X} \|x\|\ud {\mu}^N(x,s)\stackrel{(*)}\leq 2(\Var M)_t,
\end{align*}
where $(*)$ holds by \eqref{eq:proofoflemmaonvariationsWDSpdmart}.
\end{proof}

\begin{proof}[Proof of Theorem \ref{thm:analougeofProp3.5.4}] The proof has the same structure as the proof of \cite[Proposition~3.5.16]{HNVW1}. Fix $M\in \mathcal M^p_{X}$ and $\lambda >0$. Let $K:= \|T\|_{\mathcal L(\mathcal M^p_X)}$, $M = M^1 + M^2 + M^3$ be Gundy's decomposition of $M$ from Theorem \ref{thm:Gundyconttime} at the level $\alpha \lambda$ for some $\alpha>0$ which we will fix later. Notice that all $M^1$, $M^2$ and $M^3$ are local $L^p$-martingales by Remark \ref{rem:localLpintegrabilityGundysdec}. Then
 \begin{multline}\label{eq:GundydecintheproofofWDSop}
    \mathbb P (\|(TM)^*_{\infty}\|>\lambda) \\
    \leq \mathbb P(\|(TM^1)^*_{\infty}\|>\tfrac{\lambda}{2}) + \mathbb P (\|(TM^2)^*_{\infty}\|>0) + \mathbb P(\|(TM^3)^*_{\infty}\|>\tfrac{\lambda}{2}).
 \end{multline}
Let us estimate each of these three terms separately. First,
\begin{align*}
 \mathbb P(\|(TM^1)^*_{\infty}\|>\tfrac{\lambda}{2}) &\stackrel{(i)}\leq  \Bigl(\frac 2{\lambda}\Bigr)^{p} \mathbb E\|(TM^1)^*_{\infty}\|^p \stackrel{(*)}\leq \Bigl(\frac 2{\lambda} \frac{p}{p-1}\Bigr)^{p} \mathbb E\|(TM^1)_{\infty}\|^p\\ &\stackrel{(ii)}\leq \Bigl( \frac{2K}{\lambda}\frac{p}{p-1}\Bigr)^{p} \mathbb E\|M^1_{\infty}\|^p
 \leq \Bigl( \frac{2K}{\lambda}\frac{p}{p-1}\Bigr)^{p} \|M^1_{\infty}\|^{p-1}_{\infty}\mathbb E\|M^1_{\infty}\|\\ &\stackrel{(iii)}\leq\Bigl( \frac{2K}{\lambda}\frac{p}{p-1}\Bigr)^{p} (2\alpha\lambda)^{p-1} 5\mathbb E \|M_{\infty}\|
 =\frac{5\bigl(4\alpha K\frac{p}{p-1}\bigr)^p}{2\alpha \lambda}\mathbb E \|M_{\infty}\|,
\end{align*}
where $(i)$ follows from \eqref{eq:1,inftyineqforsubmart}, $(*)$ follows from Doob's maximal inequality \cite[Theorem 1.3.8(iv)]{KS}, $(ii)$ holds by the definition of $K$, and $(iii)$ follows from Gundy's decomposition.

Now turn to $M^2$. By \eqref{eq:starconditionforthm:analougeofProp3.5.4}
\begin{equation}\label{eq:fromstarconditionforthm:analougeofProp3.5.4tooperator}
  \mathbb P((TM^2)^*_{\infty}>0)\leq \mathbb P((M^2)^*_{\infty}>0)\leq \frac{4}{\alpha\lambda}\mathbb E \|M_{\infty}\|.
\end{equation}
Finally, by Lemma \ref{lem:WDSvariations} and the fact that $TM^3\stackrel{w}\ll M^3$ we have that 
$$
\mathbb E (\Var TM^3)_{\infty}\leq 2 \mathbb E (\Var M^3)_{\infty},
$$
hence
\begin{align*}
 \mathbb P(\|(TM^3)^*_{\infty}\|>\tfrac{\lambda}{2}) &\stackrel{(i)}\leq \frac{2}{\lambda}\mathbb E \|(TM^3)^*_{\infty}\| \leq \frac{2}{\lambda} \mathbb E (\Var TM^3)_{\infty}\\
 &\stackrel{(ii)}\leq \frac{4}{\lambda}\mathbb E (\Var M^3)_{\infty}
 \stackrel{(*)}\leq \frac{28}{\lambda}\mathbb E \|M_{\infty}\|,
\end{align*}
where $(i)$ follows from \eqref{eq:1,inftyineqforsubmart}, $(ii)$ holds by \eqref{eq:fromstarconditionforthm:analougeofProp3.5.4tooperator}, and $(*)$ holds by Theorem~\ref{thm:Gundyconttime}(iii).
Therefore by \eqref{eq:GundydecintheproofofWDSop}
\begin{align*}
 \lambda  \mathbb P (\|(TM)^*_{\infty}\|>\lambda) &\leq \lambda\Bigl(\frac{5\bigl(4\alpha K\frac{p}{p-1}\bigr)^p}{2\alpha \lambda} +\frac{4}{\alpha\lambda} + \frac{28}{\lambda}\Bigr) \mathbb E\|M_{\infty}\|\\
 &=\Bigl(\frac{5\bigl(4\alpha K\frac{p}{p-1}\bigr)^p}{2\alpha} +\frac{4}{\alpha} +28\Bigr) \mathbb E\|M_{\infty}\|,
\end{align*}
and by choosing $\alpha = \frac{p-1}{4Kp}$ we get
\begin{align*}
  \lambda  \mathbb P (\|(TM)^*_{\infty}\|>\lambda) &\leq \Bigl(10K\frac{p}{p-1} +16K\frac{p}{p-1}+28\Bigr) \mathbb E\|M_{\infty}\| \\
 &= \Bigl(26K\frac{p}{p-1}+28\Bigr) \mathbb E\|M_{\infty}\|,
\end{align*}
which is exactly \eqref{eq:estforL^1L^1inftynormthrL^pnorm}.
\end{proof}

The following proposition shows that the operator $T$ from Theorem \ref{thm:analougeofProp3.5.4} has a~special structure given the filtration $\mathbb F = (\mathcal F_t)_{t\geq 0}$ is generated by $(\mathcal F_n)_{n\geq 0}$: such martingale transforms are the same as those considered in \cite[Proposition 3.5.4]{HNVW1} and \cite{Burk81}.

\begin{proposition}\label{prop:discretefiltrationopProp3.5.4}
 Let $X$ be a separable Banach space. Let the filtration $\mathbb F = (\mathcal F_t)_{t\geq 0}$ be of the following form: $\mathcal F_t = \mathcal F_{\lfloor t\rfloor}$ for each $t\geq 0$, $T$ be as in Theorem \ref{thm:analougeofProp3.5.4}. Then there exists an $(\mathcal F_n)_{n\geq 0}$-predictable sequence $(a_{n})_{n\geq 0}$ with values in $[-1,1]$ such that $\Delta (TM)_n = a_n \Delta M_n$ a.s.\ for each $n\geq 0$ for any $M \in \mathcal M_X^p$.
\end{proposition}

\begin{proof}
 Let $\mathbb G = (\mathcal G_n)_{n\geq 0} := (\mathcal F_n)_{n\geq 0}$ be a discrete filtration. Due to the construction of $\mathbb F$ and the fact that $\mathbb G$ is discrete we have that any $\mathbb F$-martingale $M$ is in fact discrete (i.e.\ $M_t = M_{\lfloor t\rfloor}$ a.s.\ for each $t\geq 0$), hence any martingale has accessible jumps, so by Lemma~\ref{lem:WDS0appearsearlierforN} it is sufficient to use the fact that $TM \stackrel{w}\ll M$ for any $M \in \mathcal M^p_X$ in order to apply Theorem~\ref{thm:analougeofProp3.5.4}. Let us show that there exists a $\mathbb G$-adapted \mbox{$[-1,1]$-va}\-lued sequence $(a_n)_{n\geq 1}$ such that $\Delta (TM)_n = a_n \Delta M_n$ a.s.\ for each $n\geq 0$.
 Since $X$ is separable, $L^p(\Omega;X)$ is separable by \cite[Proposition 1.2.29]{HNVW1}. Let $(\xi^m)_{m\geq 1}$ be a dense subset of $L^p(\Omega;X)$. For each $m\geq 1$ we construct a martingale $M^m$ in the following way: $M^m_t := \mathbb E (\xi^m|\mathcal F_t)$, $t\geq 0$. Then we have that $((TM)^m_n)_{n\geq 0}\stackrel{w}\ll (M^m_n)_{n\geq 0}$ for each $m \geq 1$, so by \cite{Y17FourUMD} for each $m\geq 1$ there exists a $\mathbb G$-adapted $[-1,1]$-valued sequence $(a_n^m)_{n\geq 0}$ such that $\Delta (TM^m)_n = a_n^m \Delta M^m_n$ for each $n\geq 0$. Let us show that for each $m_1\neq m_2$ and $n\geq 0$ we have that
 \begin{equation}\label{eq:a^m_1_n=a^m_2_n}
    a^{m_1}_n = a^{m_2}_n \;\;\; \text{a.s. on} \;\; A_n^{m_1,m_2},
 \end{equation}
 where $A_n^{m_1,m_2} := \{\Delta M^{m_1}_n \neq 0\}\cap\{\Delta M^{m_2}_n \neq 0\}$.
Let $\bigl((c_1^k, c_2^k)\bigr)_{k\geq 1}$ be a dense subset of $\mathbb R^2$ such that for each $k\geq 1$
$$
c_1^k \Delta M^{m_1}_n + c_2^k\Delta M^{m_2}_n \neq 0\;\;\; \text{a.s. on} \;\; A_n^{m_1,m_2}.
$$
Then $T(c_1^k M^{m_1} + c_2^k M^{m_2}) \stackrel{w}\ll c_1^k M^{m_1} + c_2^k M^{m_2}$ for each $k\geq 1$, and hence by the linearity of $T$ we have that for each $k\geq 1$ a.s.\ $c_1^k a^{m_1}_n\Delta M^{m_1}_n + c_2^k a^{m_2}_n\Delta M^{m_2}_n$ and $c_1^k \Delta M^{m_1}_n + c_2^k \Delta M^{m_2}_n$ are collinear vectors in $X$, and 
\[
 \biggl| \frac{c_1^k a^{m_1}_n\Delta M^{m_1}_n + c_2^k a^{m_2}_n\Delta M^{m_2}_n}{c_1^k \Delta M^{m_1}_n + c_2^k \Delta M^{m_2}_n} \biggr| \leq 1\;\;\; \text{a.s. on} \;\; A_n^{m_1,m_2},
\]
by the weak differential subordination.
Therefore we can redefine $A_n^{m_1,m_2}$ up to a~negligible set in the following way:
\begin{align*}
  A_n^{m_1,m_2} := A_n^{m_1,m_2}&\bigcap_{k\geq 1} \{c_1^k \Delta M^{m_1}_n + c_2^k\Delta M^{m_2}_n \neq 0\}\\
  &\bigcap_{k\geq 1}\biggl\{\biggl| \frac{c_1^k a^{m_1}_n\Delta M^{m_1}_n + c_2^k a^{m_2}_n\Delta M^{m_2}_n}{c_1^k \Delta M^{m_1}_n + c_2^k \Delta M^{m_2}_n} \biggr| \leq 1\biggr\}.
\end{align*}

Let us now fix any $\omega \in A_n^{m_1,m_2}$ and $\eps>0$. Let $x^*\in X^*$ be such that $\langle \Delta M^{m_1}_n(\omega), x^*\rangle \neq 0$ and $\langle \Delta M^{m_2}_n(\omega), x^*\rangle \neq 0$ (such $x^*$ exists by the Hahn-Banach theorem and the definition of $A_n^{m_1,m_2}$). Then we can find $k\geq 1$ such that 
\begin{equation}\label{eq:goodc_1^kc_2^k}
 0<\frac{\langle c_1^k \Delta M^{m_1}_n(\omega) + c_2^k \Delta M^{m_2}_n(\omega), x^*\rangle}{|c_1^k| + |c_2^k|}<\eps
\end{equation}
since $\bigl((c_1^k, c_2^k)\bigr)_{k\geq 1}$ is dense in $\mathbb R^2$ (i.e.\ $k\geq 0$ such that $(c_1^k, c_2^k)$ is almost orthogonal to $(\langle  \Delta M^{m_1}_n(\omega)x^*\rangle,\langle\Delta M^{m_2}_n(\omega), x^*\rangle)$). But on the other hand (we will omit $\omega$ for the convenience of the reader)
\begin{align}\label{eq:1geqalmostinfty-1}
 1&\geq \biggl| \frac{c_1^k a^{m_1}_n\Delta M^{m_1}_n + c_2^k a^{m_2}_n\Delta M^{m_2}_n}{c_1^k \Delta M^{m_1}_n + c_2^k \Delta M^{m_2}_n} \biggr| = \frac{|\langle c_1^k a^{m_1}_n\Delta M^{m_1}_n + c_2^k a^{m_2}_n\Delta M^{m_2}_n, x^*\rangle|}{\langle c_1^k \Delta M^{m_1}_n + c_2^k \Delta M^{m_2}_n, x^*\rangle}\nonumber\\
 &=\frac{|\langle c_2^k (a^{m_2}_n - a^{m_1}_n)\Delta M^{m_2}_n, x^*\rangle|-|\langle c_1^k a^{m_1}_n\Delta M^{m_1}_n + c_2^k a^{m_1}_n\Delta M^{m_2}_n, x^*\rangle|}{\langle c_1^k \Delta M^{m_1}_n + c_2^k \Delta M^{m_2}_n, x^*\rangle}\\
 &\stackrel{(*)}\geq |a^{m_2}_n - a^{m_1}_n| |\langle\Delta M^{m_2}_n, x^*\rangle| \frac{1}{\eps} - 1,\nonumber
\end{align}
where $(*)$ holds by the triangle inequality, \eqref{eq:goodc_1^kc_2^k}, and the fact that $|a^{m_1}_n|\leq 1$. Since $\eps$ was arbitrary, \eqref{eq:1geqalmostinfty-1} holds true if and only if $a^{m_2}_n(\omega) - a^{m_1}_n(\omega) = 0$. Now since $\omega\in A_n^{m_1,m_2}$ was arbitrary, $a^{m_1}_n = a^{m_2}_n$ on $A_n^{m_1,m_2}$.

Now we define for each $n\geq 0$ and $m\geq 1$:
\begin{align*}
 B_n^1 &= \{\Delta M_n^1 \neq 0\},\\
 B_n^m &= \{\Delta M_n^m \neq 0\}\setminus B_n^{m-1}, \;\;\; m\geq 2,\\
 B_n^0 &= \Omega \setminus \bigcup_{m\geq 1} B_n^m,
\end{align*}
and define $a_n$ in the following way:
\begin{equation}\label{eq:a_nconstructionforproponProposition 3.5.4}
 \begin{split}
   a_n(\omega) &:= a_n^m, \;\;\; \omega \in B_n^m,\; m\geq 1,\\
 a_n(\omega) &:= 0, \;\;\; \omega \in B_0^m.
 \end{split}
\end{equation}
Then by \eqref{eq:a^m_1_n=a^m_2_n} $a_n = a_n^m$ a.s.\ on $\{\Delta M_n^m \neq 0\}$ for all $m\geq 1$. Therefore $\Delta (TM^m)_n = a_n \Delta M^m_n$ a.s.\ for all $m\geq 1$. Now let $M$ be a general $L^p$-martingale. Let $(M^{m_k})_{k\geq 1}$ be a sequence which converges to $M$ in $\mathcal M_{X}^p$. Fix $n\geq 0$. Then by \cite[Corollary~2.6.30]{HNVW1} $\Delta M_n^{m_k}$ converges to $\Delta M_n$ in $L^p(\Omega; X)$ as $k\to \infty$, so by boundedness of $a_n$ we have that $a_n \Delta M_n^{m_k} \to a_n \Delta M_n$ in $L^p(\Omega; X)$. On the other hand by boundedness of $T$ and by \cite[Corollary~2.6.30]{HNVW1}
\[
 \lim_{k\to \infty}a_n \Delta M_m^{n_k} = \lim_{k\to \infty} \Delta (TM_n^{m_k})_n = \Delta(TM)_n,
\]
where the limit is taken in $L^p(\Omega; X)$. Hence $\Delta(TM)_n = a_n \Delta M_n$ a.s.

It follows from \eqref{eq:a_nconstructionforproponProposition 3.5.4} and \cite{Y17FourUMD} that $(a_n)_{n\geq 0}$ is $\mathbb G$-adapted and bounded by $1$. Now let us show that $(a_n)_{n\geq0}$ is $\mathbb G$-predictable. Assume the opposite. Then there exists $N\geq 0$ such that $a_N$ is $\mathcal F_N$-measurable, but not $\mathcal F_{N-1}$-measurable (here we set $\mathcal F_{-1}$ to be the $\sigma$-algebra generated by all negligible sets). Fix $x\in X\setminus \{0\}$. Then we can construct the following $L^p$-martingale $M:\mathbb R_+ \times \Omega \to X$: $\Delta M_n = 0$ if $n\neq N$ and $\Delta M_N = (a_N-\mathbb E(a_N|\mathcal F_{N-1}))x$. This is an $L^p$-martingale since by the triangle inequality and \cite[Theorem 34.2]{Bill95}
\begin{align*}
 \|a_N-\mathbb E(a_N|\mathcal F_{N-1})\|_{\infty} &\leq \|a_N\|_{\infty} + \|\mathbb E(a_N|\mathcal F_{N-1})\|_{\infty} \leq 1 + \|\mathbb E(|a_N||\mathcal F_{N-1})\|_{\infty}\\
 &\leq 1 + \|\mathbb E(1|\mathcal F_{N-1})\|_{\infty}\leq 2.
\end{align*}
Then we have that $\Delta(TM)_N = a_N(a_N-\mathbb E(a_N|\mathcal F_{N-1}))x$, and since $TM$ is a martingale,
\begin{align*}
 0=\mathbb E (\Delta(TM)_N |\mathcal F_{N-1}) &= \mathbb E (a_N(a_N-\mathbb E(a_N|\mathcal F_{N-1}))x|\mathcal F_{N-1})\\
 &=x \mathbb E (a_N^2-a_N\mathbb E(a_N|\mathcal F_{N-1})|\mathcal F_{N-1})\\
 &= x\Bigl(\mathbb E (a_N^2|\mathcal F_{N-1}) - \bigl(\mathbb E (a_N|\mathcal F_{N-1})\bigr)^2\Bigr)\\
 &= x\mathbb E \Bigl(\bigl(a_N - \mathbb E (a_N|\mathcal F_{N-1})\bigr)^2\Big|\mathcal F_{N-1}\Bigr),
\end{align*}
so since $x\neq 0$ and the fact that $\bigl(a_N - \mathbb E (a_N|\mathcal F_{N-1})\bigr)^2$ is nonnegative we get that $a_N - \mathbb E (a_N|\mathcal F_{N-1}) = 0$ a.s., hence $a_N$ is $\mathcal F_{N-1}$-measurable.
\end{proof}

\begin{remark}
 One can extend Proposition \ref{prop:discretefiltrationopProp3.5.4} to the case of a Banach space $X$ being over the scalar field $\mathbb C$. The point is that because of the structure of the filtration $\mathbb F$ any $\mathbb F$-martingale is purely discontinuous, so one can extend the definition of weak differential subordination in the way presented in \cite{Y17UMD^A}; namely, $N \stackrel{w}\ll M$ if $|\langle\Delta N_t, x^*\rangle|\leq |\langle\Delta M_t, x^*\rangle|$ a.s.\ for all $t\geq 0$ and $x^* \in X^*$. Then by applying the same proof one can show that the sequence $(a_{n})_{n\geq 0}$ from Proposition~\ref{prop:discretefiltrationopProp3.5.4} exists and is still $(\mathcal F_n)_{n\geq 0}$-predictable, but it takes values in the unit disk $\mathbb D := \{\lambda \in \mathbb C: |\lambda|
 \leq 1\}$.
\end{remark}

\subsection{Sufficiency of the UMD property}\label{subsec:suffofUMD}

Now we will consider two examples of an operator $T$ from Theorem \ref{thm:analougeofProp3.5.4}, which will provide us with the Meyer-Yoeurp and the Yoeurp decompositions of any UMD space-valued local martingale.

\begin{theorem}[Meyer-Yoeurp decomposition of local martingales]\label{thm:MYdecoflocmartsufofUMD}
 Let $X$ be a~UMD Banach space, $M:\mathbb R_+\times \Omega \to X$ be a local martingale. Then there exist unique local martingales $M^c, M^d:\mathbb R_+ \times \Omega \to X$ such that $M^c$ is continuous, $M^d$ is purely discontinuous, $M^c_0=0$, and $M = M^c + M^d$. Moreover, for any $\lambda >0$ and $t\geq 0$
 \begin{equation}\label{eq:L^1inftyforthm:MYdecoflocmartsufofUMD}
  \begin{split}
   \lambda \mathbb P ((M^c)^*_t >\lambda) &\lesssim_X \mathbb E \|M_t\|,\\
    \lambda \mathbb P ((M^d)^*_t >\lambda) &\lesssim_X \mathbb E \|M_t\|.
  \end{split}
 \end{equation}
\end{theorem}

For the proof we will need the following lemma.

\begin{lemma}\label{lem:pdmartingalesisaclosedsbsinL^1}
 Let $M:\mathbb R_+ \times \Omega \to X$ be an $L^1$-martingale, $(M^n)_{n\geq 1}$ be a sequence of purely discontinuous $X$-valued $L^1$-martingales such that $M^n_{\infty}\to M_{\infty}$ in $L^1(\Omega;X)$. Then $M$ is purely discontinuous.
\end{lemma}

\begin{proof}
 Without loss of generality $M_0 = 0$ and $M^n_0=0$ a.s.\ for each $n\geq 1$. By Proposition \ref{thm:purdiscorthtoanycont1} it is sufficient to check that $MN$ is a martingale for any bounded continuous real-valued martingale $N$ with $N_0=0$ a.s. Fix such $N$. Then due to Proposition~\ref{thm:purdiscorthtoanycont1} $M^n N$ is a martingale for each $n\geq 0$. Moreover, since $N_t$ is bounded for each $t\geq 0$, $(M^n N)_t \to (MN)_t$ in $L^1(\Omega;X)$. Therefore by the boundedness of a conditional expectation operator (see \cite[Corollary 2.6.30]{HNVW1}) for each $0\leq s\leq t$
 \begin{align*}
   \mathbb E((MN)_t|\mathcal F_s) = \mathbb E \bigl(\lim_{n\to \infty}(M^nN)_t|\mathcal F_s\bigr) &= \lim_{n\to \infty}\mathbb E((M^nN)_t|\mathcal F_s)\\
   &=\lim_{n\to \infty}(M^nN)_s = (MN)_s.
 \end{align*}
Hence, $MN$ is a martingale. Since $N$ was arbitrary, $M$ is a purely discontinuous martingale.
\end{proof}

\begin{proof}[Proof of Theorem \ref{thm:MYdecoflocmartsufofUMD}]
 By a stopping time argument we can assume that $M$ is an $L^1$-martingale. Fix $p\in(1,\infty)$. Let $(M^n)_{n\geq 1}$ be a sequence of $X$-valued $L^p$-martingales such that $M^n_{\infty}\to M_{\infty}$ in $L^1(\Omega; X)$. Without loss of generality assume that $\mathbb E\|M_{\infty}-M^n_{\infty}\|<\frac{1}{2^{n+1}}$ for each $n\geq 1$. Let $T\in \mathcal L(\mathcal M^p_{X})$ be such that $T$ maps an $L^p$-martingale $N:\mathbb R_+ \times \Omega \to X$ to its continuous part $N^c$ (such an operator exists and bounded by Proposition \ref{prop:candecofUMDspacevalL^pmart}). For each $n\geq 1$ we denote $TM^n$ by $M^{n,c}$. Then we know that by Theorem \ref{thm:analougeofProp3.5.4} for each $m\geq n\geq 1$ and any $K>0$
 \begin{equation}\label{eq:estfromMYdecM^nM^m}
  \mathbb P((M^{n,c}-M^{m,c})^*_{\infty}>K)\lesssim_{p,X}\frac{1}{K}\mathbb E \|M^{n,c}_{\infty}-M^{m,c}_{\infty}\|\leq \frac{1}{2^{n}K},
 \end{equation}
hence $(M^{n,c})_{n\geq 1}$ is a Cauchy sequence in the ucp topology by \eqref{eq:ucpconvergencehow}. Notice that all the $M^{n,c}$'s are continuous local martingales, which are complete in the ucp topology (see \cite[pp.\ 7--8]{VY2016} and Lemma \ref{lem:supofcontfuncwithlimitiscont}). Hence there exists a local martingale $M^c:\mathbb R_+ \times \Omega \to X$ which is the limit of $(M^{n,c})_{n\geq 1}$ in the ucp topology. Now it is sufficient to prove that $M^c_0=0$ and that $\langle M-M^c, x^*\rangle$ is a purely discontinuous local martingale for any $x^* \in X^*$ in order to show that $M^c$ is the desired continuous local martingale. Firstly, $M^c_0 = \mathbb P-\lim_{n\to \infty}M^{n,c}_0 =0$ since $M^c$ is the limit of $(M^{n,c})_{n\geq 1}$ in the ucp topology and since $M^{n,c}_0 =0$ a.s.\ for each $n\geq 1$. Secondly, since $M^{n,c}\to M^c$ in the ucp topology and $M^n\to M$ in $L^1(\Omega; X)$, $\langle M^n-M^{n,c}, x^*\rangle\to \langle M-M^c, x^*\rangle$ in the ucp topology for each fixed $x^* \in X^*$. Without loss of generality set that $\mathbb E\|M_{\infty}\|, \mathbb E \|M^n_{\infty}\|\leq 1$ for each $n\geq 1$. Also by choosing a subsequence we can assume that $M^{c,n}\to M^{c}$ a.s.\ uniformly on compacts. Therefore by Lemma \ref{lem:supofcontfuncwithlimitiscont} the process $t\mapsto \sup_{0\leq s\leq t}\sup_n \|M^{c,n}\|$ exists and continuous, and for each $m\geq 1$ we can define a stopping time $\tau_m$ in the following way
$$
\tau_m := \inf\bigl\{t\geq 0: \sup_{0\leq s\leq t}\sup_n \|M^{c,n}\| \geq m\bigr\}.
$$
Notice that a.s.\ $\tau_m \to \infty$ as $m\to \infty$. First show that $\langle (M-M^c)^{\tau_m}, x^*\rangle$ is purely discontinuous for each $m\geq 1$. Note that $(M^{c,n})^{\tau_m}_{\infty} \to (M^c)^{\tau_m}_{\infty}$ and $(M^n)^{\tau_m}_{\infty} \to M^{\tau_m}_{\infty}$ in $L^1(\Omega; X)$ as $n\to \infty$. Therefore 
$$
\langle (M^n-M^{c,n})^{\tau_m}, x^*\rangle \to\langle (M-M^c)^{\tau_m}, x^*\rangle
$$
in $L^1(\Omega)$, so by Lemma \ref{lem:pdmartingalesisaclosedsbsinL^1} $\langle (M-M^c)^{\tau_m}, x^*\rangle$ is purely discontinuous. Notice that by letting $m$ to infinity we get that $\langle M-M^c, x^*\rangle$ is a purely discontinuous local martingale for any $x^* \in X^*$, hence $M-M^c$ is a purely discontinuous local martingale.

The uniqueness of the decomposition follows from Remark \ref{rem:MY==weakMY}, while \eqref{eq:L^1inftyforthm:MYdecoflocmartsufofUMD} holds due to the limiting argument, \eqref{eq:estfromMYdecM^nM^m}, and the completeness of $L^{1,\infty}$-spaces provided by (1.1.11) and Theorem 1.4.11 in \cite{GrafCl}.
\end{proof}

Let us turn to the Yoeurp decomposition.

\begin{theorem}[Yoeurp decomposition of local martingales]\label{thm:YdecoflocmartsufofUMD}
 Let $X$ be a UMD Banach space, $M^d:\mathbb R_+\times \Omega \to X$ be a purely discontinuous local martingale. Then there exist unique purely discontinuous local martingales $M^q, M^a:\mathbb R_+ \times \Omega \to X$ such that $M^q$ is quasi-left continuous, $M^a$ has accessible jumps, $M^q_0=0$, and $M^d = M^q + M^a$. Moreover, for any $\lambda >0$ and $t\geq 0$
 \begin{equation}\label{eq:L^1inftyforthm:YdecoflocmartsufofUMD}
  \begin{split}
   \lambda \mathbb P ((M^q)^*_t >\lambda) &\lesssim_X \mathbb E \|M^d_t\|,\\
    \lambda \mathbb P ((M^a)^*_t >\lambda) &\lesssim_X \mathbb E \|M^d_t\|.
  \end{split}
 \end{equation}
\end{theorem}

For the proof we will need the following lemmas.

\begin{lemma}\label{lem:starstaffvsqvofajmart}
 Let $M:\mathbb R_+ \times \Omega \to \mathbb R$ be a local martingale with accessible jumps, $M_0=0$ a.s. Then $\{M^*_{\infty} = 0\} = \{[M]_{\infty} = 0\}$ up to a negligible set.
\end{lemma}

\begin{proof}
 Let $M=M^c + M^q + M^a$ be the canonical decomposition of $M$ (see Subsection \ref{subsec:decompositionsofmartingales}). Then $M^q=0$ since $M$ has accessible jumps. By \cite[Exercise 17.3]{Kal} $\{(M^c)^*_{\infty} = 0\} = \{[M^c]_{\infty} = 0\}$ up to a negligible set. Let us show the same for $M^a$. Let $\tau:= \inf\{t\geq 0: \Delta M^a_t \neq 0\}$ be a stopping time. Notice that a.s.\
 \begin{align*}
  \{\tau<\infty\} &\subset \Bigl\{\sum_{t\geq 0}|\Delta M^a_t| > 0\Bigr\} \subset \{(M^a)^*_{\infty} > 0\},\\
  \{\tau<\infty\} &\subset  \Bigl\{\sum_{t\geq 0}|\Delta M^a_t|^2 > 0\Bigr\} = \{[M^a]_{\infty} > 0\},
 \end{align*}
so we can redefine $M^a := (M^a)^{\tau}$. By the definition of $\tau$ we have that for each $t\geq 0$ a.s.\
$\sum_{0\leq s\leq t }|\Delta M^a_s| = |\Delta M^a_{\tau}| \mathbf 1_{\tau \leq t}$, hence by \cite[Theoreme (1-6).3]{Yoe76} a.s.
\begin{equation}\label{eq:Yforforaccjumps}
 M^a_t = \Delta M^a_{\tau} \mathbf 1_{\tau \leq t},\;\;\; t\geq 0.
\end{equation}
Therefore since $[M^a]_t = |\Delta M^a_{\tau}|^2 \mathbf 1_{\tau \leq t}$ we have that $\{(M^a)^*_{\infty} = 0\} = \{[M^a]_{\infty} = 0\}$ up to a negligible set.

Let us now show the desired. First notice that by \cite[Corollary 26.16]{Kal} a.s.\
\begin{equation}\label{eq:[M]infty=0thesameas[M^c]infty=0[M^a]infty=0}
  \{[M]_{\infty} = 0\} =  \{[M^c]_{\infty} + [M^a]_{\infty} = 0\} = \{[M^c]_{\infty} = 0\}\cap\{[M^a]_{\infty} = 0\}.
\end{equation}
On the other hand a.s.\
\begin{align*}
 \{M^*_{\infty} = 0\} &= \{M^*_{\infty} = 0\}\cap\{\Delta M_{t} = 0\;\forall t\geq 0\} \stackrel{(i)}= \{M^*_{\infty} = 0\}\cap\{(M^a)_{\infty}^* = 0\}\\
 &\stackrel{(ii)}=\{(M^c)^*_{\infty} = 0\}\cap\{(M^a)_{\infty}^* = 0\} \stackrel{(iii)}= \{[M^c]_{\infty} = 0\}\cap\{[M^a]_{\infty} = 0\}\\
 &\stackrel{(iv)}=\{[M]_{\infty} = 0\},
\end{align*}
where $(i)$ holds by \eqref{eq:Yforforaccjumps}, $(ii)$ follows from the fact that $M^c = M-M^a$, $(iii)$ follows from the first half of the proof, and finally $(iv)$ follows from \eqref{eq:[M]infty=0thesameas[M^c]infty=0[M^a]infty=0}.
\end{proof}

\begin{lemma}\label{lem:candecforsratprocessesandinfty}
 Let $M:\mathbb R_+ \times \Omega \to \mathbb R$ be a local martingale, $M = M^c + M^q + M^a$ be the canonical decomposition. Then up to a negligible set 
 \begin{equation}\label{eq:lem:candecforsratprocessesandinfty}
  \{M^*_{\infty} = 0\} = \{(M^c)^*_{\infty} = 0\}\cap\{(M^q)^*_{\infty} = 0\}\cap\{(M^a)^*_{\infty} = 0\}.
 \end{equation}
\end{lemma}

\begin{proof}
Let $N := M^c+M^a$. First notice that by Lemma \ref{lem:starstaffvsqvofajmart} and \cite[Corollary~26.16]{Kal} a.s.\
\begin{equation}\label{eq:PthatNM^cM^a=0}
 \begin{split}
   \{N^*_{\infty} = 0\} = \{[N]_{\infty} = 0\} &= \{[M^c]_{\infty} + [M^a]_{\infty} = 0\}\\
 &=  \{[M^c]_{\infty}=0\} \cap\{ [M^a]_{\infty} = 0\}\\
 &= \{(M^c)^*_{\infty} = 0\}\cap\{(M^a)^*_{\infty} = 0\}.
 \end{split}
\end{equation}
Let $\tau:= \inf\{t\geq 0: \Delta M_t \neq 0\}$ be a stopping time. Then a.s.\
\[
 \{\tau<\infty\} \subset  \{M^*_{\infty}>0\} \subset \{N^*_{\infty}>0\}\cup\{(M^q)^*_{\infty}>0\}
\]
since $M = N + M^d$. Let $A = \{M_{\infty}^* = 0\}\subset \Omega$. Then $[M]_{\infty} = [N+ M^q]_{\infty}=0$ a.s.\ on $A$, and consequently $[N]_{\infty}=0$ a.s.\ on $A$ by \cite[Corollary 26.16]{Kal}. Therefore by Lemma \ref{lem:starstaffvsqvofajmart} $N^*_{\infty}=0$ a.s.\ on $A$, so $(M^q)^*_{\infty}=0$ a.s.\ on $A$, and therefore by \eqref{eq:PthatNM^cM^a=0}
\begin{align*}
  \{M_{\infty}^* = 0\} = A &\subset \{N^*_{\infty}=0\}\cap \{(M^q)^*_{\infty}=0\}\\
  &= \{(M^c)^*_{\infty}=0\}\cap \{(M^q)^*_{\infty}=0\}\cap \{(M^a)^*_{\infty}=0\}.
\end{align*}
The converse inclusion follows from the fact that $M = N + M^q$ and \eqref{eq:PthatNM^cM^a=0}.
\end{proof}

\begin{lemma}\label{lem:WDS0appearsearlierforN}
 Let $X$ be a Banach space, $M, N:\mathbb R_+ \times \Omega \to X$ be local martingales such that $N$ has accessible jumps and $N \stackrel{w}\ll M$. Then
 \begin{equation}\label{eq:lem:onstoppingtimeforWDS}
  \mathbb P(N^*_t>0)\leq \mathbb P(M^*_t>0),\;\;\; t\geq 0.
 \end{equation}
\end{lemma}

\begin{proof}
 \eqref{eq:lem:onstoppingtimeforWDS} follows from the fact that $\{M_t^* = 0\}\subset \{N_t^* = 0\}$.
 Let $(x_n^*)_{n\geq 0}\subset X^*$ be a separating set. Then up to a negligible set 
 \begin{align*}
  \{M_t^* = 0\} &= \bigcap_{n\geq 0}\{ (\langle M ,x_n^*\rangle)_t^* = 0\},\\
  \{ N_t^* = 0\} &= \bigcap_{n\geq 0}\{(\langle N ,x_n^*\rangle)_t^* = 0\},
 \end{align*}
therefore it is sufficient to consider $X = \mathbb R$. Let $M = M^c + M^d + M^a$ be the canonical decomposition of $M$ (see Subsection \ref{subsec:decompositionsofmartingales}). By Lemma \ref{lem:candecforsratprocessesandinfty} and  \eqref{eq:PthatNM^cM^a=0}
\[
 \{ M_t^* = 0\} \subset \{(M^c + M^a)_t^* = 0\}.
\]
Moreover, by Lemma \ref{lem:starstaffvsqvofajmart}
\begin{align*}
  \{\ (M^c + M^a)_t^* = 0\} &= \{[M^c + M^a]_t = 0\} \subset \{[M]_t = 0\},\\
  \{ N_t^* = 0\} &= \{ [N]_t = 0\},
 \end{align*}
and hence since $N \ll M$,
\begin{align*}
  \{ M_t^* = 0\} \subset\{ [M]_t = 0\} &\subset\{ [N]_t = 0\}=   \{ N_t^* = 0\}.
\end{align*}
\end{proof}

\begin{proof}[Proof of Theorem \ref{thm:YdecoflocmartsufofUMD}]
 Without loss of generality assume that $M^d$ is an \mbox{$L^1$-mar}\-tin\-gale and $M^d_0=0$ a.s. We will divide the proof into two steps.
 
 {\em Step 1.} Define a stopping time $\tau=\{t\geq 0: \|M^d_t\|>\frac{1}{2}\}$. In this step we assume that $M^d = (M^d)^{\tau}$ (i.e.\ the martingale stops moving after reaching $\frac{1}{2}$, in particular after the first jump of absolute value bigger than $1$). Let $\mu^M$ be the random measure defined by \eqref{eq:defofranmeasuremu^M}, $\nu^M$ be the corresponding compensator (see Subsection \ref{subsec:randommeasures}). For each $n\geq 1$ define a stopping time
 \begin{equation}\label{eq:tau_nfromthm:YdecoflocmartsufofUMD}
   \tau_n = \inf\Bigl\{t\geq 0:\int_{[0,t]\times X}\|x\|\mathbf 1_{\|x\|>n}\ud \nu^{M^d}> 1\Bigr\},
 \end{equation}
 and a process $M^{d,n}:\mathbb R_+ \times \Omega \to X$ in the following way
 \begin{equation}\label{eq:defofM^d,nforYoeurpdec}
    M^{d,n}_t=\Bigl((M^d)^{\tau-}_t + \Delta M^d_{\tau}\mathbf 1_{\|\Delta M^d_{\tau}\|\leq n}\mathbf 1_{\tau\leq t} + \int_{[0,t]\times X}x\mathbf 1_{\|x\|>n}\ud \nu^{M^d}\Bigr)^{\tau_n-},\;\; t\geq 0,
 \end{equation}
where we define $M^{\sigma-}$ for a stopping time $\sigma$ in the same way as in \eqref{eq:defofM^tau-ingen}. First of all show that $\tau_n \to \infty$ a.s.\ as $n\to \infty$. Notice that by due to Subsection \ref{subsec:randommeasures}
\begin{equation}\label{eq:bddnessofintwrtnuMd}
 \begin{split}
   \mathbb E \int_{\mathbb R_+\times X}\|x\|\mathbf 1_{\|x\|>1}\ud \nu^{M^d} &= \mathbb E\int_{\mathbb R_+\times X}\|x\|\mathbf 1_{\|x\|>1}\ud \mu^{M^d} \leq \mathbb E\|\Delta M^d_{\tau}\|\\
 &\leq \mathbb E \|M^d_{\tau}\| + \mathbb E \|M^d_{\tau-}\| \stackrel{(*)}\leq \mathbb E \|M^d_{\infty}\| + \frac 12\stackrel{(**)}<\infty,
 \end{split}
\end{equation}
where $(*)$ follows from the fact that $M_{\tau} =M_{\infty}$ and the fact that $\|M_{\tau-}\|\leq \frac{1}{2}$ a.s., and $(**)$ holds due to the fact that $M$ is an $L^1$-martingale. Therefore
\begin{align*}
 \int_{\mathbb R_+\times X}\|x\|\mathbf 1_{\|x\|>1}\ud \nu^{M^d} < \infty\;\;\;\; \text{a.s.},
\end{align*}
so by the monotone convergence theorem a.s.\
\[
 \int_{\mathbb R_+\times X}\|x\|\mathbf 1_{\|x\|>n}\ud \nu^{M^d} \to 0,\;\;\; n\to \infty,
\]
and hence $\tau_n \to\infty$ as $n\to\infty$. 

We need to show that $M^{d,n}$ is an $L^{\infty}$-martingale for each $n\geq 1$. Clearly $M^{d,n}$ is adapted and c\`adl\`ag. It is also  a local martingale since it can be rewritten in the following form:
\begin{align*}
 M^{d,n}_t=(M^d)_t^{\tau_n-} -\int_{[0,t]\times X}x\mathbf 1_{\|x\|>n}\mathbf 1_{s<\tau_n}\ud \bar{\mu}^{M^d},\;\;\; t\geq 0,
\end{align*}
where the first term is a martingale by Lemma \ref{lem:predtaupreservemartingales}, and the second term is a local martingale by Lemma \ref{lem:intwrtrmidapdlocmart} and the fact that the process $s\mapsto \mathbf 1_{s<\tau_n}$ is predictable by \cite[Theorem 25.14]{Kal} and the predictability of $\tau_n$ (the latter follows from \eqref{eq:tau_nfromthm:YdecoflocmartsufofUMD} and the predictability of $\nu^{M^d}$, see Subsection \ref{subsec:randommeasures}). Moreover, for each fixed $t\geq 0$ we have that a.s.
\begin{align*}
 \|M^{d,n}_t\| &\leq \|(M^d)_t^{\tau-\wedge \tau_n-}\| + \|\Delta M^d_{\tau}\mathbf 1_{\|\Delta M^d_{\tau}\|\leq n}\| + \int_{[0,\tau_n)\times X}\|x\|\mathbf 1_{\|x\|>n}{\nu}^{M^d}\\
 &\leq 1+n+1=n+2.
\end{align*}
(Recall that $\tau-\wedge \tau_n-:= (\tau\wedge\tau_n)-$, see \eqref{eq:tau-wedgesigma-:=(tauwedgesigma)-}).
Therefore $(M^{d,n})_{n\geq 1}$ are bounded martingales.

Now let us now show that $M^{d,n}_{\infty}\to M^d_{\infty}$ in $L^{1}(\Omega; X)$. First, $M^{d,n}_{\infty} = M^{d,n}_{\tau_n-}$ a.s., so by the triangle inequality
\begin{align*}
 \mathbb E \|M^d_{\infty} - M^{d,n}_{\infty}\| \leq \mathbb E\|M^d_{\infty} - M^d_{\tau_n-}\| + \mathbb E \|M^d_{\tau_n-} - M^{d,n}_{\tau_n-}\|.
\end{align*}
Notice that the first term vanishes as $n\to \infty$ by the fact that $\|M^d_{\infty} - M^d_{\tau_n-}\| \leq 1+\|\Delta M_{\tau}\|$ a.s., the fact that $\tau_n \to \infty$ a.s., and the dominated convergence theorem. Let us consider the second term:
\begin{align*}
 \mathbb E& \|M^d_{\tau_n-} - M^{d,n}_{\tau_n-}\|\\
 &= \mathbb E \Bigl\|M^d_{\tau_n-} - (M^d)^{\tau-}_{\tau_n-}-\Delta M^d_{\tau}\mathbf 1_{\|\Delta M^d_{\tau}\|\leq n}\mathbf 1_{\tau< \tau_n} - \int_{[0,\tau_n)\times X}x\mathbf 1_{\|x\|>n}\ud \nu^{M^d} \Bigr\|\\
 &= \mathbb E \Bigl\|\Delta M^d_{\tau}\mathbf 1_{\tau< \tau_n}-\Delta M^d_{\tau}\mathbf 1_{\|\Delta M^d_{\tau}\|\leq n}\mathbf 1_{\tau<\tau_n} - \int_{[0,\tau_n)\times X}x\mathbf 1_{\|x\|>n}\ud \nu^{M^d} \Bigr\|\\
 &= \mathbb E \Bigl\|\Delta M^d_{\tau}\mathbf 1_{\|\Delta M^d_{\tau}\|> n}\mathbf 1_{\tau< \tau_n} - \int_{[0,\tau_n)\times X}x\mathbf 1_{\|x\|>n}\ud \nu^{M^d} \Bigr\|\\
 &= \mathbb E \Bigl\|\int_{[0,\tau_n)\times X}x\mathbf 1_{\|x\|>n}\ud \mu^{M^d}  - \int_{[0,\tau_n)\times X}x\mathbf 1_{\|x\|>n}\ud \nu^{M^d} \Bigr\|\\
 &\leq \mathbb E  \int_{[0,\tau_n)\times X}\|x\|\mathbf 1_{\|x\|>n}\ud \mu^{M^d}   +  \mathbb E\int_{[0,\tau_n)\times X}\|x\|\mathbf 1_{\|x\|>n}\ud \nu^{M^d}\\
 &\stackrel{(*)}= 2\mathbb E  \int_{[0,\tau_n)\times X}\|x\|\mathbf 1_{\|x\|>n}\ud \mu^{M^d} \stackrel{(**)}= 2\mathbb E \|\Delta M^d_{\tau}\| \mathbf 1_{\|\Delta M^d_{\tau}\|>n},
\end{align*}
and the last expression vanishes as $n\to \infty$ by the monotone convergence theorem. (Notice that $(*)$ follows from the definition of a compensator and from \eqref{eq:bddnessofintwrtnuMd}, while $(**)$ follows from the fact that $\|\Delta M_t\| \geq 1$ only if $t=\tau$ by the assumptions on $M$.)

Since each of $M^{d,n}$'s is an $L^p$-martingale for each $p\in (1,\infty)$, by Proposition \ref{prop:candecofUMDspacevalL^pmart} for each $n\geq 1$ there exists the Yoeurp decomposition $M^{d,n} = M^{q,n} + M^{a,n}$ of a martingale $M^{d,n}$ into a sum of two purely discontinuous martingales $M^{q,n}, M^{a,n}:\mathbb R_+ \times \Omega \to X$ such that $M^{q,n}$ is quasi-left continuous, $M^{a,n}$ has accessible jumps, and $M^{q,n}_0=M^{a,n}_0=0$ a.s.\ (recall that $M^{d,n}_0=0$ a.s.). Fix some $p\in(1,\infty)$. Since an operator $T^q$ that maps an $L^p$-martingale $M:\mathbb R_+ \times \Omega \to X$ to its purely discontinuous quasi-left continuous part $M^q$ of the canonical decomposition is continuous on $\mathcal L(\mathcal M^p_X)$ by Proposition \ref{prop:candecofUMDspacevalL^pmart}, Theorem \ref{thm:analougeofProp3.5.4} together with Lemma \ref{lem:WDS0appearsearlierforN} yields that for each $m,n\geq 1$ and $K>0$
\begin{align*}
 \mathbb P\bigl((M^{q,n}-M^{q,m})^*_{\infty}>K\bigr)& \lesssim_{p} \frac{1}{K}\mathbb E \|M^{d,n}_{\infty} - M^{d,m}_{\infty}\|\\ 
 &\leq \frac{1}{K}(\mathbb E \|M^{d,n}_{\infty} - M^{d}_{\infty}\| + \mathbb E \|M^{d,m}_{\infty} - M^{d}_{\infty}\|),
\end{align*}
so $(M^{q,n})_{n\geq 1}$ is a Cauchy sequence in the ucp topology. By Proposition \ref{prop:Xvaluedcadlagareclosedunderucp} it has a c\`adl\`ag adapted limit. Denote this limit by $M^q$. Let us show that $M^q$ is a purely discontinuous quasi-left continuous local martingale. Let $\sigma$ be a predictable time. Then $\Delta M^{q,n}_{\sigma} = 0$ a.s., and for any $t\geq 0$ a.s.\
\begin{equation}\label{eq:DeltasarethesameofM^dandM^qamdM^a}
 \begin{split}
   \sup_{0\leq s\leq t}\|M^{q,n}_s-M^q_s\|&\geq \mathbf 1_{\sigma\leq t}\sup_{0\leq s\leq \sigma}\|M^{q,n}_s-M^q_s\|\\
 &\geq \mathbf 1_{\sigma\leq t}\bigl(\sup_{m\geq 1}\|M^{q,n}_{0\vee \sigma-\frac 1m}-M^q_{0\vee \sigma-\frac 1m}\|\vee \|M^{q,n}_{\sigma}-M^q_{\sigma}\|\bigr)\\
 &\geq \frac{1}{2} \mathbf 1_{\sigma\leq t}\bigl(\limsup_{m\geq 1}\|M^{q,n}_{0\vee \sigma-\frac 1m}-M^q_{0\vee \sigma-\frac 1m}\|+ \|M^{q,n}_{\sigma}-M^q_{\sigma}\|\bigr)\\
 &=\frac{1}{2}\mathbf 1_{\sigma\leq t}\bigl(\|M^{q,n}_{\sigma-}-M^q_{\sigma-}\|+ \|M^{q,n}_{\sigma}-M^q_{\sigma}\|\bigr)\\
 &\stackrel{(*)}\geq\frac 12\mathbf 1_{\sigma\leq t}\|M^{q,n}_{\sigma-}-M^q_{\sigma-}- M^{q,n}_{\sigma}+M^q_{\sigma}\|\\
 &\geq\frac 12\mathbf 1_{\sigma\leq t}\|\Delta M^q_{\sigma-}- \Delta M^{q,n}_{\sigma}\|=\frac 12\mathbf 1_{\sigma\leq t}\|\Delta M^q_{\sigma-}\|,
 \end{split}
\end{equation}
where $(*)$ follows from the triangle inequality.
Since 
$$
\mathbb P-\lim_{n\to \infty}\sup_{0\leq s\leq t}\|M^{q,n}_s-M^q_s\| = 0,
$$ 
we have that for each $t\geq 0$
\[
 \mathbb P-\lim_{n\to \infty}\mathbf 1_{\sigma\leq t}\|\Delta M^q_{\sigma-}\| = 0.
\]
But the expression under the limit in probability does not depend on $n$. Hence $\mathbf 1_{\sigma\leq t}\|\Delta M^q_{\sigma-}\| = 0$ a.s. By letting $t\to \infty$ we get that a.s. $\|\Delta M^q_{\sigma}\| = 0$,
and since $\sigma$ was arbitrary predictable, $M^q$ is quasi-left continuous. 

Let now $\sigma$ be a totally inaccessible stopping time. Let us show that a.s.
\begin{equation}\label{eq:tistexistofYoeurpdecqlcpart}
 \Delta{M^q_{\sigma}} = \Delta M^d_{\sigma}.
\end{equation}
First notice that for each fixed $m\geq n \geq 1$
\begin{equation}\label{eq:jumpsofM^qmarethesameasofM^d}
\begin{split}
  \Delta M^{q,m}_{\sigma}\mathbf 1_{\sigma< \tau \wedge\tau_n} &\stackrel{(*)}= \Delta M^{d,m}_{\sigma}\mathbf 1_{\sigma< \tau \wedge\tau_n} \stackrel{(**)}= \Delta M^{d}_{\sigma}\mathbf 1_{\sigma< \tau \wedge\tau_n},\\
  \Delta M^{q,m}_{\sigma}\mathbf 1_{\sigma=\tau< \tau_n}\mathbf 1_{\|\Delta M^{d}_{\tau}\|\leq n} &\stackrel{(*)}= \Delta M^{d,m}_{\sigma}\mathbf 1_{\sigma=\tau< \tau_n}\mathbf 1_{\|\Delta M^{d}_{\tau}\|\leq n} \stackrel{(**)}= \Delta M^{d}_{\sigma}\mathbf 1_{\sigma=\tau< \tau_n} \mathbf 1_{\|\Delta M^{d}_{\tau}\|\leq n},
\end{split}
\end{equation}
where $(*)$ follows from Remark \ref{rem:decompsitionspreserveDeltasRvaluedcase}, and $(**)$ follows from the definition \eqref{eq:defofM^d,nforYoeurpdec} of $M^{d,m}$ and Lemma \ref{lem:DeltaV_tau=0forpredictVandtotinactau}. Therefore by \eqref{eq:DeltasarethesameofM^dandM^qamdM^a} applied for our $\sigma$ a.s.\ for each $n\geq 1$
\begin{equation}\label{eq:jumpsofM^qarethesameasofM^d}
\begin{split}
  \Delta M^{d}_{\sigma}\mathbf 1_{\sigma<\tau\wedge \tau_n} &= \Delta M^{q}_{\sigma}\mathbf 1_{\sigma<\tau\wedge \tau_n},\\
   \Delta M^{d}_{\sigma}\mathbf 1_{\sigma=\tau< \tau_n} \mathbf 1_{\|\Delta M^{d}_{\tau}\|\leq n} &= \Delta M^{q}_{\sigma}\mathbf 1_{\sigma=\tau< \tau_n} \mathbf 1_{\|\Delta M^{d}_{\tau}\|\leq n}.
\end{split}
\end{equation}
By letting $n\to \infty$ we get \eqref{eq:tistexistofYoeurpdecqlcpart}.

Let us show that $M^q$ is locally integrable. For each $l\geq 1$ set $\rho_l := \inf\{t\geq 0: \|M^q_t\|\geq l\}$. Then a.s.\ for each $t\geq 0$
\begin{align*}
 \|(M^q)^{\rho_l}_{t}\| \leq \|(M^q)^{\rho_l}_{t-}\| + \|\Delta (M^q)^{\rho_l}_{t}\| &\leq l + \|\Delta (M^q)^{\rho_l}_{t}\| \mathbf 1_{t=\tau} + \|\Delta (M^q)^{\rho_l}_{t}\| \mathbf 1_{t<\tau}\\
 &\leq l+\|\Delta M^d_{\tau}\| + 1.
\end{align*}
Therefore 
\[
 \mathbb E \|(M^q)^{\rho_l}_{t}\| \leq l+1 + \mathbb E \|\Delta M^d_{\tau}\|< \infty,
\]
where $\mathbb E \|\Delta M_{\tau}\|<\infty$ by \eqref{eq:bddnessofintwrtnuMd}. Since $M^q$ is c\`adl\`ag, by \cite[Problem V.1]{Pollard} we have that $\rho_l\to \infty$ as $l\to \infty$, so $M^q$ is locally integrable.

Now let us show that $M^q$ is a local martingale. Let $(M^{q, n_k})_{k\geq 1}$ be a subsequence of $(M^{q, n})_{n\geq 1}$ such that $M^{q, n_k}\to M^q$ uniformly on compacts a.s.\ (existence of such a subsequence can be shown e.g.\ as in the proof of \cite[Theorem 62]{Prot}). It is sufficient to show that $M^{\rho_l\wedge\tau_{n_k}-}$ is a local martingale for each $l,k\geq 1$ since $\rho_l\to \infty$ and $\tau_{n_k}\to \infty$ a.s.\ as $l,k\to \infty$. Fix $K>0$. Then by \eqref{eq:jumpsofM^qmarethesameasofM^d} and \eqref{eq:jumpsofM^qarethesameasofM^d} for each $k\geq K$ we have that a.s.\ for each $t\geq 0$
\[
 \Delta (M^{q, n_k})^{\tau_{n_K}-\wedge \tau-}_t =  \Delta (M^{q})^{\tau_{n_K}-\wedge \tau-}_t.
\]
Therefore by Lemma \ref{lem:supofcontfuncwithlimitiscont} there exists a continuous adapted process $N:\mathbb R_+ \times \Omega \to \mathbb R_+$ such that a.s.
\[
 N_t = \sup_{k\geq K} \bigl\|(M^{q, n_k})^{\tau_{n_K}-\wedge \tau-}_t - (M^{q})^{\tau_{n_K}-\wedge \tau-}_t\bigr\|,\;\;\; t\geq 0.
\]
Now for each $j\geq 1$ define a stopping time $\sigma_j = \inf\{t\geq 0: N_t\geq j\}$. Fix $j\geq 1$. Then for each $t\geq 0$ we have that for any $k\geq K$ a.s.\
$$
\bigl\|(M^{q, n_k})^{\rho_l\wedge\tau_{n_K}-\wedge \sigma_j}_t - (M^{q})^{\rho_l\wedge\tau_{n_K}-\wedge\sigma_j}_t\bigr\|\leq j+l+2\|\Delta M^d_{\tau}\|
$$
and that $(M^{q, n_k})^{\rho_l\wedge\tau_{n_K}-\wedge \sigma_j}_t - (M^{q})^{\rho_l\wedge\tau_{n_K}-\wedge\sigma_j}_t \to 0$ a.s.\ as $k\to \infty$. Hence by the dominated convergence theory
$$
(M^{q, n_k})^{\rho_l\wedge\tau_{n_K}-\wedge \sigma_j}_t\to (M^{q})^{\rho_l\wedge\tau_{n_K}-\wedge\sigma_j}_t\;\; \text{in}\, L^1(\Omega; X)\,\text{as}\,k\to \infty.
$$
Consequently, $\bigl((M^{q})^{\rho_l\wedge\tau_{n_K}-\wedge\sigma_j}_t\bigr)_{t\geq 0}$ is an $L^1$-martingale, which is moreover purely discontinuous by Lemma \ref{lem:pdmartingalesisaclosedsbsinL^1}. By letting $l,K,j\to \infty$ we get that $M^q$ is a purely discontinuous quasi-left continuous local martingale.

$M^a$ can be constructed in the same way. The identity $M^d = M^q + M^a$ follows from the following limiting argument:
\begin{align*}
 M^{d} &= ucp-\lim_{n\to \infty} M^{d,n},\\
  M^{q} &= ucp-\lim_{n\to \infty} M^{q,n},\\
   M^{a} &= ucp-\lim_{n\to \infty} M^{a,n},
\end{align*}
and the fact that $M^{d,n} = M^{q,n} + M^{a,n}$ for each $n\geq 1$.
 
 {\em Step 2.} For a general martingale $M^d$ we construct a sequence of stopping times $\tau_n = \inf\{t\geq 0:\|M^d_t\|\geq \frac{n}{2}\}$. For each $M^{d,n}:= (M^d)^{\tau_n}$ we construct the corresponding $M^{q,n}$ by Step 1. Then for each $m\geq n\geq 1$ we get that $(M^{q,n})^{\tau_m} = M^{q,m}$ since for any $x^* \in X^*$ a.s.
 \[
  \langle (M^{q,n})^{\tau_m}, x^*\rangle = \langle M^{q,m}, x^*\rangle
 \]
due to the uniqueness of the Yoeurp decomposition in the real-valued case. Then we just set $M^q_0:=0$ and
$$
M^q_t := \sum_{n\geq 1}M^{q,n}_t \mathbf 1_{t\in (\tau_{n-1}, \tau_n]},\;\;\; t\geq 0,
$$
where $\tau_0\equiv 0$. The obtained $M^q$ will be the desired purely discontinuous quasi-left continuous local martingale. 

We can construct $M^a$ in the same way and show that then $M^d = M^q + M^a$ similarly to how it was shown in step 1.

The uniqueness of the decomposition follows from Remark \ref{rem:Y==weakY}, while \eqref{eq:L^1inftyforthm:YdecoflocmartsufofUMD} follows analogously \eqref{eq:L^1inftyforthm:MYdecoflocmartsufofUMD}.
\end{proof}

\begin{proof}[Proof of Theorem \ref{thm:candecoflocmartsufandnesofUMD} (sufficiency of UMD and \eqref{eq:L^1inftyforcor:candecoflocmartsufofUMD})]
 Sufficiency of the UMD property follows from Theorem \ref{thm:MYdecoflocmartsufofUMD} and Theorem \ref{thm:YdecoflocmartsufofUMD}, while \eqref{eq:L^1inftyforcor:candecoflocmartsufofUMD} follows in the same way as \eqref{eq:L^1inftyforthm:MYdecoflocmartsufofUMD} and \eqref{eq:L^1inftyforthm:YdecoflocmartsufofUMD}.
\end{proof}

\subsection{Necessity of the UMD property}\label{subsec:NecofUMD}

In the current subsection we show that the UMD property is necessary in Theorem \ref{thm:MYdecoflocmartsufofUMD} and Theorem \ref{thm:YdecoflocmartsufofUMD}, and hence it is necessary for the canonical decomposition of a local martingale.

\begin{theorem}\label{thm:noUMDnodec}
 Let $X$ be a Banach space that does not have the UMD property. Then there exists a filtration $\mathbb F = (\mathcal F_t)_{t\geq 0}$ and an $\mathbb F$-martingale $M:\mathbb R_+ \times \Omega \to X$ such that $M$ provides neither the Meyer-Yoeurp nor the canonical decomposition.
\end{theorem}

For the proof we will need the following lemma which is a modification of the statements from p.\ 1001 and p.\ 1004 of~\cite{Burk81}. Recall that if $(f_n)_{n\geq 0}$ is an $X$-valued martingale, the we define $df_n := f_{n}-f_{n-1}$ for $n\geq 1$ and $df_0:=f_0$.

\begin{lemma}\label{lem:UMDiffg^*>1holds}
 Let $X$ be a Banach space. Then $X$ is a UMD Banach space if and only if there exists a constant $C>0$ such that for any $X$-valued discrete martingale $(f_n)_{n\geq 0}$, for any $\{0,1\}$-valued sequence $(a_n)_{n\geq 0}$ one has that
 \[
  g^*_{\infty} >1\;\; \text{a.s.}\;\;\Longrightarrow \;\;\mathbb E\|f_{\infty}\|>C,
 \]
where $(g_n)_{n\geq 0}$ is an $X$-valued discrete martingale such that $dg_n = a_ndf_n$ for each $n\geq 0$, $g^*_{\infty} := \sup_{n\geq 0} \|g_n\|$.
\end{lemma}

\begin{proof}
 One needs to modify \cite[Theorem 2.1]{Burk81} in such a way that $dg_n = a_ndf_n$ for some $a_n\in \{0,1\}$ for each $n\geq 0$. Then the proof is the same, and the desired statement follows from the equivalence of \cite[(2.3)]{Burk81} and \cite[(2.4)]{Burk81}.
\end{proof}

For the next corollary we will need to define a Rademacher random variable and a Paley-Walsh martingale. 

\begin{definition}[Rademacher random variable]
 Let $\xi:\Omega \to \mathbb R$ be a random variable. Then $\xi$ has the {\em Rademacher distribution} (or simply $\xi$ {\em is Rademacher}) if $\mathbb P (\xi = 1) = \mathbb P(\xi = -1) = \frac{1}{2}$.
\end{definition}

\begin{definition}[Paley-Walsh martingale]\label{def:Paley-Walsh}
 Let $X$ be a Banach space. A discrete $X$-valued martingale $(f_n)_{n\geq 0}$ is called a {\em Paley-Walsh martingale} if there exist a sequence of independent Rademacher random variables $(r_n)_{n\geq 1}$, a function $\phi_n:\{-1, 1\}^{n-1}\to X$ for each $n\geq 2$ and $\phi_1 \in X$ such that $df_n = r_n \phi_n(r_1,\ldots, r_{n-1})$ for each $n\geq 2$, $df_1 = r_1 \phi_1$, and $f_0$ is a constant a.s.
\end{definition}

\begin{corollary}\label{cor:UMDiffg^*>inftyholds}
 Let $X$ be a Banach space that does not have the UMD property. Then there exists an $X$-valued Paley-Walsh $L^1$-martingale $(f_n)_{n\geq 0}$ and a $\{0,1\}$-valued sequence $(a_n)_{n\geq 0}$ such that $\mathbb P(g_{\infty}^* = \infty) = 1$, where $(g_n)_{n\geq 0}$ is an $X$-valued martingale such that $dg_n = a_ndf_n$ for each $n\geq 0$.
\end{corollary}

\begin{proof}
 Without loss of generality all the martingales used below are Paley-Walsh (see \cite[Theorem 3.6.1]{HNVW1}), so the resulting martingale will be Paley-Walsh as well. By Lemma \ref{lem:UMDiffg^*>1holds} we can find $N_1>0$, an $X$-valued martingale $f^1=(f^1_n)_{n=0}^{N_1}$ and a $\{0,1\}$-valued sequence $(a_n^1)_{n=0}^{N_1}$ such that $\mathbb E\|f^1_{N_1}\|<\frac{1}{2}$ and
 \[
  \mathbb P((g^1)^*_{N_1}>1)>\frac 12,
 \]
where $g^1 = (g^1_n)_{n=0}^{N_1}$ is such that $dg^1_n = a_n^1df_n^1$ for each $n=0,\ldots,N_1$. Now inductively for each $k>1$ we find $N_k>0$ and an $X$-valued Paley-Walsh martingale $f^k=(f^k_n)_{n=0}^{N_k}$ independent of $f^1,\ldots,f^{k-1}$ such that $\mathbb E\|f^k_{N_k}\|<\frac{1}{2^k}$ and
 \[
  \mathbb P((g^k)^*_{N_k}>2C_k)>1-\frac 1{2^k},
 \]
where $g^k = (g^k_n)_{n=0}^{N_1}$ is such that $dg^k_n = a_n^kdf_n^k$ for each $n=0,\ldots,N_k$, and $C_k>2^k$ is such that 
 \[
  \mathbb P((g^1)^*_{N_1} + \ldots + (g^{k-1})^*_{N_{k-1}} >C_k)<\frac 1{2^k}.
 \]
Without loss of generality assume that $f^k_0=0$ a.s.\ for each $k\geq 1$. Now construct a martingale $(f_n)_{n\geq 0}$ and a $\{0,1\}$-valued sequence $(a_n)_{n\geq 0}$ in the following way: $f_0=a_0=0$ a.s., $df_n = df^k_m$ and $a_n = a^k_m$ if $n=N_1 + \cdots + N_{k-1} + m$ for some $k\geq 1$ and $1\leq m\leq N_k$. Then $(f_n)_{n\geq 0}$ is well-defined, 
$$
\lim_{n\to \infty} \mathbb E \|f_n\| = \mathbb E \|f_{\infty}\| \leq \sum_{k\geq 1}\mathbb E \|f_{N_k}^k\| \leq 1
$$ 
by the triangle inequality, and for an $X$-valued martingale $(g_n)_{n\geq 0}$ with $dg_n = a_ndf_n$ for each $n\geq 0$, for each $k\geq 2$
\[
\mathbb P(g^*_{N_1+ \cdots + N_k}>C_k) \geq \mathbb P((g^k)^*_{N_k}>2C_k, (g^1)^*_{N_1} + \ldots + (g^{k-1})^*_{N_{k-1}} \leq C_k) >1-\frac 1{2^{k-1}},
\]
hence $g^*_{\infty} = \infty$ a.s.\
\end{proof}

\begin{proof}[Proof of Theorem \ref{thm:noUMDnodec}]
 By Corollary \ref{cor:UMDiffg^*>inftyholds} we can construct a discrete filtration $\mathbb G =(\mathcal G_n)_{n\geq 0}$ and an $X$-valued $L^1$-integrable Paley-Walsh $\mathbb G$-martingale $(f_n)_{n\geq 0}$ such that
 \begin{equation}\label{eq:upperboundnormoffinthprfofiffUMD}
  \mathbb E\|f_{\infty}\| = \lim_{n\to \infty} \mathbb E \|f_n\| \leq 1,
 \end{equation}
and such that there exists $\{0,1\}$-valued sequence $(a_n)_{n\geq 0}$ so that
  \[
   \mathbb P(g^*_{\infty} = \infty) = 1,
  \]
where $(g_n)_{n\geq 0}$ is an $X$-valued martingale with $dg_n = a_ndf_n$ for each $n\geq 0$.

Since $(f_n)_{n\geq 0}$ is Paley-Walsh, there exist a sequence $(r_n)_{n\geq 0}$ of independent Rademacher variables, a sequence of functions $(\phi_n)_{n\geq 1}$ with $\phi_1\in X$ and $\phi_n:\{-1,1\}^{n-1} \to X$ for each $n\geq 2$, so that $df_n = r_n \phi_n(r_1, \ldots, r_{n-1})$ a.s.\ for each $n\geq 1$.

Now our goal is to construct a continuous-time $X$-valued martingale $M$ which does not have the Meyer-Yoeurp decomposition (and hence the canonical decomposition) using  $(f_n)_{n\geq 0}$. Let us first construct a filtration $\mathbb F = (\mathcal F_t)_{t\geq 0}$ on $\mathbb R_+$ in the following way. By \cite[Subsection 3.2]{Y17MartDec} for each $n\geq 0$ we can find a continuous martingale $M^n:[0, \frac{1}{2^{n+1}}]\times \Omega \to \mathbb R$ with a symmetric distribution such that $M^n_0=0$ a.s., $\bigl|M^n_{\frac{1}{2^{n+1}}}\bigr|\leq 1$ a.s., 
\begin{equation}\label{eq;proofofprop:noUMDnodecM^nisnot0a.s.}
 \mathbb P\bigl(M^n_{\frac{1}{2^{n+1}}} = 0\bigr) = 0,
\end{equation}
and
\begin{equation}\label{eq:upperboundsforsigmanotMiniffUMD}
 \mathbb P\Bigl(M^n_{\frac{1}{2^{n+1}}} \neq \sign M^n_{\frac{1}{2^{n+1}}}\Bigr) <\frac {1}{2^n (\|\phi_{n}\|_{\infty}+1)}.
\end{equation}
Let $(\tilde r_n)_{n\geq 0}$ be a sequence of independent Rademacher random variables. Without loss of generality assume that all $(\tilde r_n)_{n\geq 0}$ and $(M^n)_{n\geq 0}$ are independent. Then set $\mathcal F_0$ to be the $\sigma$-algebra generated by all negligible sets, and set
\[
 \mathcal F_t := 
 \begin{cases}
\mathcal F_{1-\frac{1}{2^n}},\;\;\; &t\in (1-\frac{1}{2^n}, 1-\frac{1}{2^{n+1}}), a_n = 0, n\geq 0,\\
\sigma(\mathcal F_{1-\frac{1}{2^n}}, \tilde r_n),\;\;\; &t= 1-\frac{1}{2^{n+1}}, a_n = 0, n\geq 0,\\
\sigma(\mathcal F_{1-\frac{1}{2^n}}, (M^n_{s})_{s\in[0, t-1-\frac{1}{2^n}]}), \;\;\; &t\in (1-\frac{1}{2^n}, 1-\frac{1}{2^{n+1}}], a_n = 1, n\geq 0,\\
 \sigma(\mathcal F_s: s\in[0,1)),\;\;\;&t\geq 1.
\end{cases}
\]
Let $(\sigma_n)_{n\geq 0}$ be a sequence of independent Rademacher variables such that $\sigma_n = \tilde r_n$ if $a_n=0$ and $\sigma_n = \sign M^n_{\frac{1}{2^{n+1}}}$ if $a_n=1$ (in the latter case $\sigma_n$ has the Rademacher distribution by \eqref{eq;proofofprop:noUMDnodecM^nisnot0a.s.} and the fact that $M^n_{\frac{1}{2^{n+1}}}$ is symmetric). Now construct $M:\mathbb R_+ \times \Omega \to X$ in the following way:
\begin{equation}\label{eq:defofMintheproofofiffUMD}
 M_t =
\begin{cases}
0, \;\;\; &t=0,\\
 M_{1-\frac{1}{2^n}},\;\;\; &t\in (1-\frac{1}{2^n}, 1-\frac{1}{2^{n+1}}), a_n = 0, n\geq 0,\\
 M_{1-\frac{1}{2^n}} + \sigma_n\phi_n(\sigma_1,\ldots,\sigma_{n-1}),\;\;\; &t= 1-\frac{1}{2^{n+1}}, a_n = 0, n\geq 0,\\
M_{1-\frac{1}{2^n}} + M^n_{t-1-\frac{1}{2^n}}\phi_n(\sigma_1,\ldots,\sigma_{n-1}), \;\;\; &t\in (1-\frac{1}{2^n}, 1-\frac{1}{2^{n+1}}], a_n = 1, n\geq 0,\\
 \lim_{n\to \infty} M_{1-\frac{1}{2^n}},\;\;\;&t\geq 1.
\end{cases} 
\end{equation}
 First we show that $\lim_{n\to \infty} M_{1-\frac{1}{2^n}}$ exists a.s., hence $M$ is well-defined. By \cite[Theorem 3.3.8]{HNVW1} it is sufficient to show that there exists $\xi\in L^1(\Omega; X)$ such that $M_{1-\frac{1}{2^n}} = \mathbb E (\xi|\mathcal F_{1-\frac{1}{2^n}})$ for all $n\geq 1$. Notice that $\bigl(M_{1-\frac{1}{2^n}}\bigr)_{n\geq 0}$ is a martingale since $M_{1-\frac{1}{2^{n+1}}} - M_{1-\frac{1}{2^n}}$ equals either $\sigma_n\phi_n(\sigma_1,\ldots,\sigma_{n-1})$ (if $a_n=0$) or $M^n_{\frac{1}{2^{n+1}}}\phi_n(\sigma_1,\ldots,\sigma_{n-1})$ (if $a_n=1$). Both random variables are bounded, and in both cases the conditional expectation with respect to $\mathcal F_{1-\frac{1}{2^n}}$ gives zero. Now let us show integrability. Let $(\tilde f_n)_{n\geq 0}$ be an $X$-valued martingale such that $\tilde f_0=0$ a.s.\ and
 \begin{equation}\label{eq:defofftildeintheproofofUMDiff}
   d\tilde f_n = \sigma_n\phi_n(\sigma_1,\ldots,\sigma_{n-1}),\;\;\; n\geq 1.
 \end{equation}
 Then $(\tilde f_n)_{n\geq 0}$ has the same distribution as $(f_n)_{n\geq 0}$, so it is $L^1$-integrable. Now fix $n\geq 1$ and let us estimate $\mathbb E \|\tilde f_n - M_{1-\frac{1}{2^n}}\|$:
 \begin{align}\label{eq:diffoff_ntildeandM_nforiffUMD}
  \mathbb E \bigl\|\tilde f_n - M_{1-\frac{1}{2^n}}\bigr\| &\stackrel{(i)}= \mathbb E \Bigl\|\sum_{k=1}^n\sigma_k\phi_k(\sigma_1,\ldots,\sigma_{k-1})\nonumber\\
  &\quad - \sum_{k=1}^n\bigl(\sigma_k\phi_k(\sigma_1,\ldots,\sigma_{k-1})\mathbf 1_{a_k=0} + M^k_{\frac{1}{2^{k+1}}}\phi_k(\sigma_1,\ldots,\sigma_{k-1})\mathbf 1_{a_k=1} \bigr)\Bigr\|\nonumber\\
  &=\mathbb E \Bigl\|\sum_{k=1}^n\bigl(\sigma_k-M^k_{\frac{1}{2^{k+1}}}\bigr)\phi_k(\sigma_1,\ldots,\sigma_{k-1}) \mathbf 1_{a_k=1} \Bigr\|\\
  &\stackrel{(ii)}\leq\sum_{k=1}^n\mathbb E \bigl\|\bigl(\sigma_k-M^k_{\frac{1}{2^{k+1}}}\bigr)\phi_k(\sigma_1,\ldots,\sigma_{k-1})  \bigr\|\nonumber\\
    &\stackrel{(iii)}\leq 2\sum_{k=1}^n\mathbb P \bigl(\sigma_k\neq M^k_{\frac{1}{2^{k+1}}}\bigr)\|\phi_k\|_{\infty} \stackrel{(iv)}\leq 2\sum_{k=1}^n \frac {1}{2^k}\leq 2,\nonumber
 \end{align}
where $(i)$ follows from \eqref{eq:defofftildeintheproofofUMDiff} and the definition of $M$ from \eqref{eq:defofMintheproofofiffUMD}, $(ii)$ holds by the triangle inequality, $(iii)$ follows from the fact that a.s.\ for each $n\geq 1$
$$
\bigl|\sigma_n - M^n_{\frac{1}{2^{n+1}}}\bigr| \leq |\sigma_n| + \bigl|M^n_{\frac{1}{2^{n+1}}}\bigr|\leq 2;
$$ 
finally, $(iv)$ follows from \eqref{eq:upperboundsforsigmanotMiniffUMD}. 
Let us show that there exists $\mathcal F_1$-measurable $\xi\in L^1(\Omega; X)$ such that $M_{1-\frac{1}{2^n}} = \mathbb E (\xi|\mathcal F_{1-\frac{1}{2^n}})$ for each $n\geq 1$. First notice that $\mathbb E (\tilde f_{\infty}|\mathcal F_{1-\frac{1}{2^n}}) = \tilde f_n$ for each $n\geq 1$. Moreover, by \eqref{eq:diffoff_ntildeandM_nforiffUMD} the series
\[
 \eta := \sum_{k=1}^{\infty}\bigl(\sigma_k-M^k_{\frac{1}{2^{k+1}}}\bigr)\phi_k(\sigma_1,\ldots,\sigma_{k-1}) \mathbf 1_{a_k=1}
\]
converges in $L^1(\Omega; X)$. Therefore, if we define $\xi := \tilde f_{\infty} - \eta$, then
\begin{align*}
 \mathbb E (\xi|\mathcal F_{1-\frac{1}{2^n}})&=\mathbb E (\tilde f_{\infty}-\eta|\mathcal F_{1-\frac{1}{2^n}})\\
 &= \tilde f_n - \mathbb E \Bigl(\sum_{k=1}^{\infty}\bigl(\sigma_k-M^k_{\frac{1}{2^{k+1}}}\bigr)\phi_k(\sigma_1,\ldots,\sigma_{k-1}) \mathbf 1_{a_k=1} | \mathcal F_{1-\frac{1}{2^n}}\Bigr)\\
 &= \tilde f_n - \sum_{k=1}^{\infty}\mathbb E \Bigl(\bigl(\sigma_k-M^k_{\frac{1}{2^{k+1}}}\bigr)\phi_k(\sigma_1,\ldots,\sigma_{k-1}) \mathbf 1_{a_k=1} | \mathcal F_{1-\frac{1}{2^n}}\Bigr)\\
 &= \tilde f_n - \sum_{k=1}^{n}\bigl(\sigma_k-M^k_{\frac{1}{2^{k+1}}}\bigr)\phi_k(\sigma_1,\ldots,\sigma_{k-1}) \mathbf 1_{a_k=1} = M_{1-\frac{1}{2^n}},
\end{align*}
so one has an a.s.\ convergence by the martingale convergence theorem \cite[Theorem~3.3.8]{HNVW1}.

Now let us show that $M$ is a martingale that does not have the Meyer-Yoeurp decomposition. Assume the contrary: let $M = M^d + M^c$ be the Meyer-Yoeurp decomposition. Then one can show that for each $n\geq 1$
\begin{align*}
M^d_{1-\frac{1}{2^n}} &= \sum_{k=1}^n \sigma_k \phi_k(\sigma_1, \ldots,\sigma_{k-1})\mathbf 1_{a_k=0},\\
M^c_{1-\frac{1}{2^n}} &= \sum_{k=1}^n M^k_{\frac{1}{2^{k+1}}} \phi_k(\sigma_1, \ldots,\sigma_{k-1})\mathbf 1_{a_k=1},
\end{align*}
by applying $x^*\in X^*$ and showing that the corresponding processes $\bigl\langle M^d_{1-\frac {1}{2^n}}, x^* \bigr\rangle$ and $\bigl\langle M^c_{1-\frac {1}{2^n}}, x^* \bigr\rangle$ are purely discontinuous and continuous local martingales respectively (see Remark \ref{rem:MY==weakMY}). Now let us show that $M^c$ is not an $X$-valued local martingale. If it is a local martingale, then 
$$
\mathbb P((M^c)^*_{\infty} = \infty) = \mathbb P((M^c)^*_{1} = \infty)=0.
$$
since $M^c$ as a local martingale should have c\`adl\`ag paths (even continuous since $M^c$ assume to be continuous). But for each fixed $n\geq 1$
\[
\mathbb P((M^c)^*_{1} = \infty) = \mathbb P\bigl(\bigl(M^c - M^c_{\frac{1}{2^{n}}}\bigr)^*_{1} = \infty\bigr) \geq \mathbb P\bigl((\tilde g-\tilde g_n)^*_{\infty} = \bigl(M^c - M^c_{\frac{1}{2^{n}}}\bigr)^*_{1} \bigr)
\]
where $(\tilde g_n)_{n\geq 0}$ is an $X$-valued martingale such that $d\tilde g_n = a_nd\tilde f_n$ a.s.\ for each $n\geq 0$, and hence by the construction in Lemma \ref{lem:UMDiffg^*>1holds} $\tilde g^*_{\infty} = \infty$ a.s. Further,
\begin{align*}
\mathbb P\bigl((\tilde g-\tilde g_n)^*_{\infty} = \bigl(M^c - M^c_{\frac{1}{2^{n}}}\bigr)^*_{1} \bigr) &= 1 - \mathbb P\bigl((\tilde g-\tilde g_n)^*_{\infty} \neq \bigl(M^c - M^c_{\frac{1}{2^{n}}}\bigr)^*_{1} \bigr)\\
&\geq 1 - \sum_{k=n}^{\infty}\mathbf 1_{a_k=1}\mathbb P\Bigl(M^k_{\frac{1}{2^{k+1}}} \neq \sigma_k\Bigr)\\
&\geq 1 - \sum_{k=n}^{\infty}\mathbb P\Bigl(M^k_{\frac{1}{2^{k+1}}} \neq \sign M^k_{\frac{1}{2^{k+1}}}\Bigr)\\
&\stackrel{(*)}\geq 1-\frac{1}{2^{n-1}},
\end{align*}
where $(*)$ follows from \eqref{eq:upperboundsforsigmanotMiniffUMD}. Since $n$ was arbitrary, $(M^c)^*_{1} = (M^c)^*_{\infty}=\infty$ a.s., so $M^c$ can not be a local martingale.
\end{proof}

\begin{proof}[Proof of Theorem \ref{thm:candecoflocmartsufandnesofUMD} (necessity of UMD)]
 Necessity of the UMD property follows from Theorem \ref{thm:noUMDnodec}.
\end{proof}

\begin{remark}
 One can also show that existence of the Yoeurp decomposition of an arbitrary $X$-valued purely discontinuous local martingale is equivalent to the UMD property. We will not repeat the argument here, but just notice that one needs to modify the proof of Theorem \ref{thm:noUMDnodec} in a way which was demonstrated in \cite[Subsection 3.2]{Y17MartDec}.
\end{remark}

\begin{remark}
 The reader might assume that one can weaken the Meyer-Yoeurp decomposition and consider a decomposition of an $X$-valued local martingale $M$ into a sum of a continuous $X$-valued {\em semi}martingale $N^c$ and a purely discontinuous \mbox{$X$-va}\-lued {\em semi}martingale $N^d$, which perhaps may happen in a broader (rather than UMD) class of Banach spaces. Then for any reasonable definition of an $X$-valued semimartingale we get that $N^c = M^c + A$ for some continuous local martingale $M^c$ and an adapted process of (weakly) bounded variation $A$. Hence $M = N^c + N^d = M^c + (N^d + A)$, where $N^d+A = M- M^c$ is a local martingale, which is purely discontinuous, so $M$ should have the Meyer-Yoeurp decomposition as well in this setting, which means that the UMD property is crucial.
\end{remark}

\bibliographystyle{plain}

\def\cprime{$'$} \def\polhk#1{\setbox0=\hbox{#1}{\ooalign{\hidewidth
  \lower1.5ex\hbox{`}\hidewidth\crcr\unhbox0}}}
  \def\polhk#1{\setbox0=\hbox{#1}{\ooalign{\hidewidth
  \lower1.5ex\hbox{`}\hidewidth\crcr\unhbox0}}} \def\cprime{$'$}

\end{document}